\documentclass{amsart}

\usepackage{amsfonts,amsthm,amsmath,amssymb,latexsym}
\usepackage{graphicx,color}
\usepackage[all]{xy}
\usepackage{import}

\graphicspath{{./Users/dragomir/Documents/Work/Projects/thurstons_bdry_2/liouville-uniform-weak*-bdry/}}

\begin{document}


\author{Dragomir \v Sari\' c}
\thanks{This research is partially supported by National Science Foundation grant DMS 1102440.}

\address{Department of Mathematics, Queens College of CUNY,
65-30 Kissena Blvd., Flushing, NY 11367}
\email{Dragomir.Saric@qc.cuny.edu}

\address{Mathematics PhD. Program, The CUNY Graduate Center, 365 Fifth Avenue, New York, NY 10016-4309}

\theoremstyle{definition}

 \newtheorem{definition}{Definition}[section]
 \newtheorem{remark}[definition]{Remark}
 \newtheorem{example}[definition]{Example}

\newtheorem*{notation}{Notation}

\theoremstyle{plain}

 \newtheorem{proposition}[definition]{Proposition}
 \newtheorem{theorem}[definition]{Theorem}
 \newtheorem{corollary}[definition]{Corollary}
 \newtheorem{lemma}[definition]{Lemma}

\def\H{{\mathbb H}}
\def\F{{\mathcal F}}
\def\R{{\mathbb R}}
\def\Q{{\mathbb Q}}
\def\Z{{\mathbb Z}}
\def\E{{\mathcal E}}
\def\N{{\mathbb N}}
\def\X{{\mathcal X}}
\def\Y{{\mathcal Y}}
\def\C{{\mathbb C}}
\def\D{{\mathbb D}}
\def\G{{\mathcal G}}

\title[Thurston's boundary]{Thurston's boundary to infinite-dimensional Teichm\"uller spaces: geodesic currents}


\date{\today}

\maketitle

\begin{abstract}
Let $X_0$ be a complete borderless infinite area hyperbolic surface. We introduce  Thurston's 
boundary to the Teichm\"uller space $T(X_0)$ of the surface $X_0$ using Liouville (geodesic) currents.  
Thurston's boundary to $T(X_0)$ is identified with the space $PML_{bdd}(X_0)$ of projective bounded measured laminations on $X_0$ which 
naturally extends Thurston's result for closed surfaces. Moreover, the quasiconformal mapping class group $MCG_{qc}(X_0)$ acts continuously on the closure $T(X_0)\cup PML_{bdd}(X_0)$.
\end{abstract}

\section{Introduction}

Fix a complete borderless infinite area hyperbolic surface $X_0$. The space of all quasiconformal deformations of $X_0$ modulo post-compositions by isometries and bounded homotopies is an infinite-dimensional Banach
manifold called the {\it Teichm\"uller space} $T(X_0)$ of $X_0$. 
A hyperbolic metric on a surface $X_0$ induces a natural Borel measure on the space of geodesics of the universal covering $\tilde{X}_0$ called the {\it Liouville current}. 
We describe limiting
behavior 
of the quasiconformal deformations of $X_0$ when dilatations of quasiconfomal 
maps increase without a bound by taking the projective limits of corresponding Liouville currents. Thurston \cite{Th1}, \cite{FLP} 
used the length spectrum to compactify the Teichm\"uller space of a closed surface of genus at least two by 
adding to it the space of projective measured laminations of the surface. Bonahon \cite{Bo} 
used Liouville currents to embed 
the Teichm\"uller space of a closed surface of genus at least two into the space of geodesic currents and give an alternative description of Thurston's boundary to the Teichm\"uller space of a closed surface of genus at least two. We use Bonahon's setup in our construction of Thurston's boundary to infinite dimensional Teichm\"uller spaces.

The Teichm\"uller space $T(X_0)$ of an infinite area hyperbolic surface $X_0$ is an infinite dimensional non-separable Banach manifold. In order to make the map from $T(X_0)$ into the space of geodesic currents of $X_0$ an embedding for the Teichm\"uller metric, some care is needed when defining a topology on the space of geodesic currents of $X_0$.
In \cite{Sar}, H\"older topology on the space of geodesic currents of an infinite area hyperbolic surface $X_0$ is introduced in order to give a natural definition of Thurston's boundary to the Teichm\"uller space $T(X_0)$ of an infinite area hyperbolic surface $X_0$. Thurston's boundary is identified with the space $PML_{bdd}(X_0)$ of projective bounded measured laminations on $X$ analogous to the case of closed surfaces (cf. \cite{Sar}). 

H\"older topology on the space of geodesic currents is given by a family of $\nu$-norms for  H\"older exponents $0<\nu\leq 1$ (cf. \cite{Sar}). This is somewhat complicated description of a topology that could prevent further applications of Thurston's boundary.
Our main contribution is an improvement in the choice of the topology on the space of geodesic currents of $X_0$. 
Namely, we adopt the uniform weak* topology (cf. \cite{MS}) to the space of geodesic 
currents and prove that Thurston's boundary to $T(X_0)$ is identified with 
$PML_{bdd}(X_0)$ as before (cf. \cite{Sar}).

\vskip .2 cm

Let $X_0$ be a complete, borderless hyperbolic surface of (possibly) infinite area (e.g. the hyperbolic plane $\mathbb{H}$, the complement of a Cantor set in the Riemann sphere, a topologically finite hyperbolic surface with funnel ends, an infinite genus surface). The hyperbolic plane $\mathbb{H}$ is identified with the unit disk model and the visual boundary of $\mathbb{H}$ is identified with the unit circle $S^1$. The universal covering $\tilde{X}_0$ is isometrically identified with the hyperbolic plane $\mathbb{H}$ and the isometry continuously extends to an identification of the boundary at infinity $\partial_{\infty}\tilde{X}_0$ with the unit circle $S^1$. The space $G(\tilde{X}_0)$ of oriented geodesics of $\tilde{X}_0$ is identified with $(\partial_{\infty}\tilde{X}_0\times\partial_{\infty}\tilde{X}_0)-diag\equiv (S^1\times S^1)-diag$ by assigning to each geodesic the pair of its endpoints, where $diag$ is the diagonal of $S^1\times S^1$. 

The set $[a,b]\times [c,d]\subset (S^1\times S^1)-diag$ is called a {\it box of geodesics}, where $[a,b],[c,d]\subset S^1$ are disjoint closed arcs. The {\it Liouville measure} of the box of geodesic  $[a,b]\times [c,d]$ is (cf. \cite{Bo})
$$
L([a,b]\times [c,d])=\log \frac{(a-c)(b-d)}{(a-d)(b-c)}.
$$
If $A\subset (S^1\times S^1)-diag$ is a Borel set, then the Liouville measure of $A$ is given by
$$
L(A)=\int_A\frac{|dx|\cdot |dy|}{|x-y|^2}.
$$
The identification of $G(\tilde{X}_0)$ with $(S^1\times S^1)-diag$ induces a full support, $\pi_1(X_0)$-invariant Borel measure on $G(\tilde{X}_0)$ via the pull-back of the Liouville measure on $(S^1\times S^1)-diag$. We remark that $X_0$ is required to be borderless and complete since the Liouville measure is naturally defined on $S^1\times S^1-diag$.

Two different hyperbolic metrics on $X_0$ induce different identifications of $G(\tilde{X}_0)$ and $(S^1\times S^1)-diag$ which in turn induce different measures on the space of geodesics $G(\tilde{X}_0)$ via pull-backs of the Liouville measure. 
Denote by $\mathcal{M}(G(\tilde{X}_0))$ the space of all positive Borel measures (called {\it geodesic currents}) on $G(\tilde{X}_0)$. The {\it Liouville map} 
$$
\mathcal{L}:T(X_0)\to \mathcal{M}(G(\tilde{X}_0))
$$
is defined by assigning to each marked hyperbolic metric the pull-back of the Liouville measure under the identification of $\tilde{X}_0$ and $\mathbb{H}^2$ induced by the hyperbolic metric (cf. Bonahon \cite{Bo}).

When $X_0$ is a finite closed surface of genus at least two, Bonahon \cite{Bo} proved that the Liouville map is a homeomorphism onto its image when $\mathcal{M}(G(\tilde{X}_0))$ is equipped with the weak* topology. Moreover, the projectivization $P(\mathcal{L}(T(X_0)))$  of the image $\mathcal{L}(T(X_0))$ under the Liouville map remains a homeomorphism onto its image in the space of projective geodesic currents $P(\mathcal{M}(G(\tilde{X}_0)))$. Bonahon \cite{Bo} proved that the boundary of $P(\mathcal{L}(T(X_0)))$
inside  $P(\mathcal{M}(G(\tilde{X}_0)))$ consists of projective measured laminations $PML(X_0)$ of the closed surface $X_0$ thus giving an alternative description of Thurston's boundary to $T(X_0)$.

From now on, we assume that $X_0$ is a hyperbolic surface of infinite area.  A positive Borel measure $m$ on $G(\tilde{X}_0)$, called a geodesic current, is said to be {\it bounded}
if 
$$
\sup_{[a,b]\times [c,d]}m([a,b]\times [c,d])<\infty
$$
where the supremum is over all boxes of geodesics $[a,b]\times [c,d]$ with $L([a,b]\times [c,d])=\log 2$. 
Denote by $\mathcal{M}(G(\tilde{X}_0))$ the space of bounded geodesic currents on $G(\tilde{X}_0)$. The Liouville map $\mathcal{L}:T(X_0)\to \mathcal{M}(G(\tilde{X}_0))$ is injective. If $\mathcal{M}(G(\tilde{X}_0))$ is equipped with the weak* topology then the Liouville map is not a homemorphism onto its image. In \cite{Sar},  a new topology on $\mathcal{M}(G(\tilde{X}_0))$ is introduced by embedding $\mathcal{M}(G(\tilde{X}_0))$ into the space of H\"older distributions on $G(\tilde{X}_0)$ satisfying certain boundedness conditions. The Liouville map is an analytic homeomorphism onto its image in the space of H\"older distributions (cf. Otal \cite{Ot}, and also \cite{Sar3}).

The H\"older topology on $\mathcal{M}(G(\tilde{X}_0))$ is used to introduce Thurston's boundary to the Teichm\"uller space $T(X_0)$ when $X_0$ is a hyperbolic surface of infinite area (cf. \cite{Sar}). It turns out that Thurston's boundary to $T(X_0)$ is the space of all projective bounded measured laminations $PML_{bdd}(X_0)$ of $X_0$ analogous to the case of closed surfaces. Unlike for closed surfaces, Thurston's bordification $T(X_0)\cup PML_{bdd}(X_0)$ is not compact, in fact it is not even locally compact. 

The H\"older topology on $\mathcal{M}(G(\tilde{X}_0))$ is complicated for applications. The purpose of this paper is to give a simpler topology on $\mathcal{M}(G(\tilde{X}_0))$ while obtaining same Thurston's boundary to $T(X_0)$.  The topology on $\mathcal{M}(G(\tilde{X}_0))$ that we use is called the {\it uniform weak* topology} and it is first introduced on the space $ML_{bdd}(\tilde{X}_0)$ in \cite{MS} for the purposes of studying the relationship between the earthquake measures and hyperbolic structures obtained by the corresponding earthquakes. We somewhat simplify the definition of the uniform weak* topology from \cite{MS}.

A sequence of measures $m_k\in \mathcal{M}(G(\tilde{X}_0))$ converges to $m\in \mathcal{M}(G(\tilde{X}_0))$ as $k\to\infty$ in the {\it uniform weak* topology} if for every continuous function $f:G(\tilde{X}_0)\to \mathbb{R}$ with compact support we have 
$$
\sup_{\gamma\in Isom(\tilde{X}_0)}\Big{|}\int_{G(\tilde{X}_0)}fd[\gamma^{*}(m_k-m)]\Big{|}  \to 0
$$
as $k\to\infty$, where the supremum is over all isometries $\gamma$ of $\tilde{X}_0=\mathbb{H}$.
In other words, all pull-backs of $m_k-m$ by isometries must converge at the same speed to zero when integrated against a continuous function with compact support. The ``uniformity'' comes from the fact that we consider pull-backs over all isometries in the supremum. We obtain

\vskip .2 cm

\noindent {\bf Theorem 1. } {\it Let $X_0$ be a complete hyperbolic surface without border with possibly infinite area. Then the Liouville map
$$
\mathcal{L}:T(X_0)\to \mathcal{M}(G(\tilde{X}_0))
$$
is a homeomorphism onto its image when $\mathcal{M}(G(\tilde{X}_0))$ is equipped with the uniform weak* topology. The image $\mathcal{L}(T(X_0))$ is closed and unbounded in $\mathcal{M}(G(\tilde{X}_0))$.

The projectivization
$$
P\mathcal{L}:T(X_0)\to P(\mathcal{M}(\tilde{X}_0))
$$
of the Liouville map is a homeomorphisms and the image $P(\mathcal{L}(T(X_0)))$ is not closed in 
$P(\mathcal{M}(\tilde{X}_0))$. The boundary of $P(\mathcal{L}(T(X_0)))$ is the space $PML_{bdd}(X_0)$ of projective bounded measured laminations- Thurston's boundary to $T(X_0)$.
} 

\vskip .2 cm

\noindent {\bf Remark.} When $X_0$ is a closed surface of genus at least two, then the weak* topology coincides with the uniform weak* topology on the space of geodesic currents of $X_0$. The reason for this is that geodesic currents are invariant under the action of $\pi_1(X_0)$ which is a cocompact Fuchsian group.

\vskip .2 cm

In the course of proving Theorem 1 we establish

\vskip .2 cm

\noindent {\bf Theorem 2. } {\it Let $\beta\in ML_{bdd}(X_0)$ and let $t\mapsto E^{t\beta}|_{S^1}$ for $t>0$ be an earthquake path in $T(X_0)$ with the earthquake measure $t\beta$. Then
$$
\frac{1}{t}(E^{t\beta}|_{S^1})^{*}(L)\to\beta
$$
as $t\to\infty$, where the convergence is in the uniform weak* topology.} 

\vskip .2 cm

The {\it quasiconformal mapping class group} $MCG_{qc}(X_0)$ of a complete borderless infinite area hyperbolic surface $X_0$ consists of all quasiconformal maps $g:X_0\to X_0$ up to bounded homotopy (cf. \cite{GL}). The natural action of $MCG_{qc}(X_0)$ on $T(X_0)$ is continuous in the Teichm\"uller metric. We prove

\vskip .2 cm

\noindent {\bf Theorem 3. } {\it Let $X_0$ be a complete hyperbolic surface without border with possibly infinite area. The action of $MCG_{qc}(X_0)$ on $T(X_0)$ extends to a continuous action on Thurston's bordification $T(X_0)\cup PML_{bdd}(X_0)$.
}

\vskip .2 cm

\noindent {\it Acknowledgements.} Theorem 2 did not appear in the first version of this paper. We thank anonymous referee for pointing out this to us.

\section{Teichm\"uller spaces of geometrically infinite hyperbolic surfaces}

Let $X_0$ be a complete hyperbolic surface without boundary whose area is infinite. The universal covering $\tilde{X}_0$ of the surface $X_0$ is isometrically identified with the hyperbolic plane $\mathbb{H}$. The boundary at infinity $\partial_{\infty}\tilde{X}_0$ is identified with the unit circle $S^1$.

The {\it Teichm\"uller space} $T(X_0)$ of the surface $X_0$ is the space of equivalence classes of all quasiconformal maps $f:X_0\to X$, where $X$ is an arbitrary complete hyperbolic surface modulo an equivalence relation. Two quasiconformal maps $f_1:X_0\to X_1$ and $f_2:X_0\to X_2$ are {\it equivalent} if there exists an isometry $I:X_1\to X_2$ such that $f_2^{-1}\circ I\circ f_1$ is homotopic to the identity under a bounded homotopy. Denote by $[f]\in T(X_0)$ the equivalence class of a quasiconformal map $f:X_0\to X$.

The {\it Teichm\"uller distance} on $T(X_0)$ is defined by
$$
d_{T}([f_1],[f_2])=\frac{1}{2}\log \inf_{g\simeq f_2\circ f_1^{-1}}K(g)
$$
where the infimum is taken over all quasiconformal maps $g$ homotopic to $f_2\circ f_1^{-1}$ and $K(g)$ is the quasiconformal constant of $g$. The {\it Teichm\"uller topology} on $T(X_0)$ is the topology induced by the Teichm\"uller distance.

Let $f:X_0\to X$ be a quasiconformal map. Denote by $\tilde{f}:\mathbb{H}\to\mathbb{H}$ a lift of $f$ to the universal covering. Then $\tilde{f}:\mathbb{H}\to\mathbb{H}$ extends by continuity to a quasisymmetric map $h:S^1\to S^1$ that conjugates the covering group of $X_0$ onto the covering group of $X$. We normalize $h$ to fix $1$, $i$ and $-1$ by post-composing it with an isometry of $\mathbb{H}$, if necessary.

Recall that $h:S^1\to S^1$ is a {\it quasisymmetric map} if it is an orientation 
preserving homeomorphism and there exists $M\geq 1$ such that
$$
\frac{1}{M}\leq \Big{|}\frac{h(e^{i(x+t)})-h(e^{ix})}{h(e^{ix})-h(e^{i(x-t)})}\Big{|}
\leq M
$$ 
for all $x,t\in\mathbb{R}$. 

The Teichm\"uller space $T(X_0)$ is in a one to one correspondence with the space of quasisymmetric maps of $S^1$ that fix $1$, $i$ and $-1$, and that conjugate the covering group of $X_0$ onto a subgroup of the isometry group of $\mathbb{H}$. From this point on, we consider the Teichm\"uller space $T(X_0)$ to be the space of normalized quasisymmetric maps.
A sequence $h_n\in T(X_0)$ {\it converges in the Teichm\"uller topology} to $h\in T(X_0)$ if
$$
\sup_{x,t\in\mathbb{R}}  \Big{|}\frac{h_n\circ h^{-1}(e^{i(x+t)})-h_n\circ h^{-1}(e^{ix})}{h_n\circ h^{-1}(e^{ix})-h_n\circ h^{-1}(e^{i(x-t)})}\Big{|}\to 0
$$
as $n\to\infty$. 

The {\it universal Teichm\"uller space} $T(\mathbb{H})$ is the Teichm\"uller space of the hyperbolic plane $\mathbb{H}$ and it consists of all normalized quasisymmetric maps of $S^1$ without any requirements on conjugating covering groups because $\mathbb{H}$ is simply connected. The universal Teichm\"uller space $T(\mathbb{H})$ contains multiple copies of Teichm\"uller spaces of all hyperbolic surfaces. In what follows, we mainly work with $T(\mathbb{H})$ since all the constructions, arguments and statements remain true under the conjugation requirement.

\section{Measured laminations and earthquakes}

A {geodesic lamination} on a hyperbolic surface $X$ is a closed subset of $X$ that is foliated by mutually non-intersecting, simple, complete geodesics called {\it leaves} of the lamination. A geodesic lamination on $X$ lifts to a geodesic lamination on $\mathbb{H}$ that is invariant under the action  of the covering group of $X$. A {\it stratum} of a geodesic lamination is either a leaf of the lamination or a connected component of the complement. A connected component of the complement of a geodesic lamination in $\mathbb{H}$ is isometric to a possibly infinite sided geodesic polygon whose sides are complete geodesics and possibly arcs on $S^1$. 

A {\it measured lamination} $\mu$ on $X$ is an assignment of a positive Borel measure on each arc transverse to a geodesic lamination $|\mu |$ that is invariant under homotopies relative leaves of $|\mu |$. The geodesic lamination $|\mu |$ is called the {\it support} of $\mu$. A measured lamination on $X$ lifts to a measured lamination on $\mathbb{H}$ that is invariant under the covering group of $X$.

 A {\it left earthquake} $E:X_0\to X$ with support geodesic lamination $\lambda$ is a surjective map that is isometry on each stratum of $\lambda$ such that each stratum is moved to the left relative to any other stratum. 
An earthquake of $X_0$ lifts to an earthquake of $\mathbb{H}$ where the support is the lift of the support on $X_0$ (cf. Thurston \cite{Th1}).

We give a definition of a (left) earthquake $E:\mathbb{H}\to\mathbb{H}$ with support geodesic lamination $\lambda$ on $\mathbb{H}$. A {\it (left) earthquake} $E:\mathbb{H}\to\mathbb{H}$ is a bijection of $\mathbb{H}$ whose restriction to any stratum of $\lambda$ is an isometry of $\mathbb{H}$; if $A$ and $B$ are two strata of $\lambda$ then
$$
(E|_A)^{-1}\circ E|_B 
$$
is a hyperbolic translation whose axis weakly separates $A$ and $B$ that moves $B$ to the left as seen from $A$ (cf. Thurston \cite{Th1}).

An earthquake $E:\mathbb{H}\to\mathbb{H}$ induces a transverse measure $\mu$ to its support $\lambda$ which defines a measured lamination $\mu$ with $|\mu |=\lambda$ (cf. \cite{Th1}). 
An earthquake of $\mathbb{H}$ extends by continuity to a homeomorphism of $S^1$. Thurston's earthquake theorem states that any homeomorphism of $S^1$ can be obtained by continuous extension of a left earthquake (cf. Thurston \cite{Th1}).

Given a measured lamination $\mu$, there exists a map $E^{\mu}:\mathbb{H}\to\mathbb{H}$ whose transverse measure is $\mu$ and that satisfies all properties in the definition of an earthquake of $\mathbb{H}$ except being onto (cf. \cite{Th1}, \cite{GHL}). $E^{\mu}$ is uniquely determined by $\mu$ up to post-composition by an isometry of $\mathbb{H}^2$.

We define {\it Thurston's norm} of a measured lamination $\mu$ as
$$
\|\mu\|_{Th}=\sup_J\mu (J)
$$
where the supremum is over all hyperbolic arcs $J$ of length $1$.

Since we are working with quasisymmetric maps, we consider measured 
laminations whose earthquakes induces quasisymmetric maps of $S^1$. 
An earthquake $E^{\mu}$ extends by continuity to a quasisymmetric map of $S^1$ if and only if $\|\mu
\|_{Th}<\infty$ (cf. \cite{Th1}, \cite{GHL}, \cite{Sa1}, \cite{Sa2}).

Denote by $ML_{bdd}(\mathbb{H})$ the space of all measured laminations on $\mathbb{H}$
with finite Thurston's norm. The above result gives a bijective map
$$
EM:T(\mathbb{H})\to ML_{bdd}(\mathbb{H})
$$ 
defined by 
$$
EM:h\mapsto \mu
$$
where $\mu$ is measured lamination induced by unique earthquake $E:\mathbb{H}\to\mathbb{H}$ whose continuous extension to $S^1$ equals $h$. 

Note that $\| t\mu\|_{Th}=t\|\mu\|_{Th}$, for $t>0$. Then, for $\|\mu\|_{Th}<\infty$, we have that
the earthquake path $t\mapsto E^{t\mu}|_{S^1}$, for $t>0$, defines a path
of quasisymmetric maps, which is a path in $T(\mathbb{H})$ when the maps are normalized
to fix $1$, $i$ and $-1$.

\section{Liouville measure, geodesic currents and uniform weak* topology}

Let $G(\mathbb{H})$ be the space of oriented complete geodesics in
the hyperbolic plane $\mathbb{H}$. Each oriented geodesic is determined by a pair of its two ideal endpoints on $S^1$ which gives
$$
G(\mathbb{H})\cong S^1\times S^1-diag
$$
where $diag$ is the diagonal in $S^1\times S^1$. If $[a,b],[c,d]\subset S^1$ are disjoint closed arcs, then the set $[a,b]\times [c,d]$ is called a {\it box of geodesics}. 

The Liouville measure on $G(\mathbb{H})$ is given by
$$
L(A)=\int_A\frac{dtds}{|e^{it}-e^{is}|^2}
$$
for any Borel set $A\subset G(\mathbb{H})$. If $A=[a,b]\times [c,d]$, then we have
$$
L([a,b]\times [c,d])=|\log\frac{(c-a)(d-b)}{(d-a)(c-b)}|.
$$
In other words, the Liouville measure of a box of geodesics is the logarithm of a 
cross-ratio of the four endpoints defining the box. Consequently, the Liouville measure is invariant under isometries of $\mathbb{H}$ and under the $\mathbb{Z}_2$-action that changes the orientation of geodesics.

A {\it geodesic current} $\alpha$ is a positive Borel measure on $G(\mathbb{H})$. Define the {\it supremum norm} of $\alpha$ by
$$
\|\alpha\|_{\sup}=\sup_{L(Q)=\log 2}\alpha (Q)
$$
The space $\mathcal{M}(G(\mathbb{H}))$ consists of all geodesic currents with finite supremum norm.

Note that a measured lamination is a geodesic currents whose support is a geodesic lamination. If a measured lamination has finite Thurston's norm then it has finite supremum norm. Thus
$$
ML_{bdd}(\mathbb{H})\subset \mathcal{M}(G(\mathbb{H})).
$$

We define the uniform weak* topology on $\mathcal{M}(G(\mathbb{H}))$ which will be used to introduce Thurston's boundary to Teichm\"uller spaces of infinite surfaces. The uniform weak* topology (in an equivalent form) was introduced in \cite{MS} on the space $ML_{bdd}(\mathbb{H})$.

\begin{definition} A sequence $\alpha_n\in\mathcal{M}(G(\mathbb{H}))$ converges to $\alpha\in \mathcal{M}(G(\mathbb{H}))$ in the {\it uniform weak* topology} if for any continuous $f:G(\mathbb{H})\to\mathbb{R}$ with compact support we have
$$
\sup_{\gamma\in Isom(\mathbb{H})}\int_{G(\mathbb{H})}fd[\gamma^{*}(\alpha_n-\alpha )]\to 0
$$
as $n\to\infty$, where $Isom(\mathbb{H})$ is the space of isometries of $\mathbb{H}$.
\end{definition} 

An equivalent definition of the uniform weak* topology was first given on $ML_{bdd}(\mathbb{H})$ (cf. \cite{MS}). The main result in \cite{MS} is that the earthquake measure map
$$
EM:T(\mathbb{H})\to ML_{bdd}(\mathbb{H})
$$
is a homeomorphism for the uniform weak* topology on $ML_{bdd}(\mathbb{H})$. In other words, 
the uniform weak* topology is a natural topology on measured laminations which makes 
correspondence between quasisymmetric maps and their earthquake measures bi-continuous.

\section{Embedding of Teichm\"uller space into geodesic currents space}

We define a map from the universal Teichm\"uller space $T(\mathbb{H})$ into the space of geodesic currents $\mathcal{M}(G(\mathbb{H}))$. Namely, the {\it Liouville map}
$$
\mathcal{L}: T(\mathbb{H})\to \mathcal{M}(G(\mathbb{H}))
$$
is given by the pull-back
$$
\mathcal{L}(h)=h^{*}L
$$
where $h\in T(\mathbb{H})$. 

\begin{theorem}
\label{thm:embedding}
The Liouville map
$$
\mathcal{L}:T(\mathbb{H})\to\mathcal{M}(G(\mathbb{H}))
$$
is a homeomorphism onto its image, where $\mathcal{M}(G(\mathbb{H}))$ is equipped with the uniform weak* topology. In addition, $\mathcal{L}(T(\mathbb{H}))$ is closed and unbounded subset of $\mathcal{M}(G(\mathbb{H}))$.
\end{theorem}

\begin{proof}
We first establish that $\mathcal{L}$ is injective. Indeed, $h\in T(\mathbb{H})$ is normalized to fix $1,i,-1\in S^1$. For $x\in S^1-\{ 1,i,-1\}$, denote by $Q_x$ a box of geodesics whose defining intervals on $S^1$ have endpoints $1,i,-1$ and $x$. Then $L(h(Q_x))$ uniquely determines $h(x)$. Thus $\mathcal{L}$ is injective.

We prove that $\mathcal{L}$ is continuous. Consider $h_n\to h$ in $T(\mathbb{H})$. Let $f:G(\mathbb{H})\to \mathbb{R}$ be a continuous function with compact support in $G(\mathbb{H})$. Define $\mathcal{L}(h_n)=\alpha_n$ and $\mathcal{L}(h)=\alpha$. 

To estimate 
$$
\Big{|}\int_{G(\mathbb{H})}fd[\gamma^{*}(\alpha_n-\alpha )]\Big{|},
$$
we cover the support of $f$ by finitely many boxes of geodesics $\{ Q_i\}^m_{i=1}$ with disjoint interiors such that

$$
L(Q_i)\leq \log 2
$$
and
$$
|\max_{Q_i}f-\min_{Q_i}f|<\epsilon_0
$$
for all $1\leq i\leq m$ and fixed $\epsilon_0$ to be determined later. The number of boxes $m$ depends on $f$ and $\epsilon_0$.

Let 
$$
s=\sum_{i=1}^m(\max_{Q_i}f)\chi_{Q_i}
$$
be a simple function approximating $f$ and let $\gamma\in Isom(\mathbb{H})$.

Then
$$
\Big{|}\int_{G(\mathbb{H})}(f-s)d[\gamma^{*}(\alpha_n-\alpha)] 
\Big{|}\leq
\epsilon_0 \sum_{i=1}^m(\alpha (Q_i)+\alpha_n(Q_i))\leq 3\epsilon_0 \sum_{i=1}^m\alpha (Q_i)$$
where the second inequality holds for all $n\geq n_0$ with $n_0$ large enough such that $h_n$ is close enough to $h$ in $T(\mathbb{H})$ (cf. Lemma \ref{lem:LMeas_Teich_neigh}).

By using Lemma \ref{lem:LMeas_Teich_neigh} again, 
$$
\Big{|}\int_{G(\mathbb{H})}sd[\gamma^{*}(\alpha_n-\alpha )]
\Big{|}\leq\epsilon\max |f|\cdot m
$$
for all $n\geq n_1$, where $n_1=n_1(\delta ,\epsilon )$ is large enough such that $h_n\in N(h,\delta ,\epsilon )$ with $\delta=\min_i L(Q_i)$. 

By choosing $\epsilon_0$ and $\epsilon$ small enough, the quantity $ \Big{|}\int_{G(\mathbb{H})}
f d[\gamma^{*}(\alpha_n-\alpha )]\Big{|}$ is as small as we want for all $n\geq\max\{ n_0,n_1\}$, where $n_0,n_1$ depend on $\epsilon_0,\epsilon ,f,m$ and do not depend on $\gamma\in Isom(\mathbb{H})$. Thus $\alpha_n\to\alpha$ as $n\to\infty$ and $\mathcal{L}$ is continuous.

We prove that $\mathcal{L}^{-1}:\mathcal{L}(T(\mathbb{H}))\to T(\mathbb{H})$ is continuous. Consider $\alpha_n\to\alpha$ in $\mathcal{M}(G(\mathbb{H}))$ with $\mathcal{L}(h_n)=\alpha_n$ and $\mathcal{L}(h)=\alpha$.

First we prove that there is an upper bound on the quasisymmetric constants of $\{ h_n\}$. Assume on the 
contrary that the quasisymmetric constants of $\{ h_n\}$ go to infinity. Then there exists a sequence of boxes $\{ Q_n\}$ 
with $L(Q_n)=\log 2$ and $\alpha_n(Q_n)\to\infty$ as $n\to\infty$.
Fix a box $Q^{*}=[1,i]\times [-1,-i]$ and let $\gamma_n\in Isom(\mathbb{H})$ be such that 
$
\gamma_{n}^{-1}\Big{(}Q_n\Big{)}=Q^{*}.
$
Let $f:G(\mathbb{H})\to\mathbb{R}$ be a non-negative continuous function with compact support such that $f|_{Q^{*}}=1$. By $\alpha_n\to\alpha$, there exists $n_0$ such that, for all $n\geq n_0$,
$$
\int_{G(\mathbb{H})}fd[(\gamma_{n})^{*}\alpha_n]\leq \int_{G(\mathbb{H})}fd[(\gamma_{n})^{*}\alpha ]+1.$$

On the other hand, 
$$
\int_{G(\mathbb{H})}fd[(\gamma_n)^{*}\alpha_n]\geq \alpha_n(Q_n)\to\infty
$$
which gives a contradiction with the above inequality. Thus quasisymmetric constants of the sequence $\{ h_n\}$ are uniformly bounded.

To prove that $h_n\to h$ in $T(\mathbb{H})$, it is enough to prove that 
$$
\sup_{L(Q)=\log 2}|\alpha_n(Q)-\alpha (Q)|\to 0
$$
as $n\to\infty$.

For a given box $Q$ with $L(Q)=\log 2$, let $\gamma_Q\in Isom(\mathbb{H})$ be such that $\gamma_Q^{-1}(Q)=Q^{*}$.
Let $Q_{\delta}$ be a sub-box of $Q$ such that
$$
\gamma_Q^{-1}(Q_{\delta})=[e^{i\delta}, e^{i(\pi /2-\delta )}]\times [e^{i(\pi +\delta)},e^{i(3\pi /2-\delta )}]\subset Q^{*}.
$$

Then $Q-Q_{\delta}$ is the union of four boxes $Q_i(\delta )$, $i=1,\ldots ,4$, such that
$L(Q_i(\delta ))\to 0$ as $\delta\to 0$ for all $i$. Since $\{ h_n\}$ is a bounded sequence in $T(\mathbb{H})$, it follows that $\alpha_n(Q_i(\delta ))\to 0$ and $\alpha (Q_i(\delta ))\to 0$ as $\delta\to 0$ uniformly in $n$. Finally, let $f_{\delta}:G(\mathbb{H})\to\mathbb{R}
$ be a positive  continuous function with $supp(f_{\delta})\subset Q^{*}$, $\| f_{\delta}\|_{\infty}=1$ and $f_{\delta}|_{[e^{i\delta}, e^{i(\pi /2-\delta )}]\times [e^{i(\pi +\delta)},e^{i(3\pi /2-\delta )}]}=1$.

It follows
$$
\Big{|}\alpha_n(Q_{\delta})-\alpha (Q_{\delta})\Big{|}\leq \Big{|}\int_{G(\mathbb{H})} f_{\delta}
d[(\gamma_Q)^{*}(\alpha_n-\alpha )]\Big{|}+ \alpha_n(Q-Q_{\delta})+\alpha (Q-Q_{\delta}).
$$
Since $ \alpha_n(Q-Q_{\delta})$ and $\alpha (Q-Q_{\delta})$ are as small as we want (uniformly in $n$) for $\delta >0$ small enough and
$$
\Big{|}\int_{G(\mathbb{H})} f_{\delta}
d((\gamma_Q)^{*}(\alpha_n-\alpha ))\Big{|}\to 0
$$
as $n\to\infty$, it follows that
$$
|\alpha_n(Q_{\delta})-\alpha (Q_{\delta})|
$$
is small for $n$ large. Thus
$$
\sup_Q|\alpha_n(Q)-\alpha (Q)|\to 0
$$
as $n\to\infty$ and $\mathcal{L}^{-1}:\mathcal{L}(T(\mathbb{H}))\to T(\mathbb{H})$ is continuous.

We prove that $\mathcal{L}(T(\mathbb{H}))$ is closed in 
$\mathcal{M}(G(\mathbb{H}))$. Indeed, let $\alpha_n\to\alpha$ in the uniform weak* topology on
$\mathcal{M}(G(\mathbb{H}))$,
where $\mathcal{L}(h_n)=\alpha_n$ for $h_n\in T(\mathbb{H})$. 
Consequently, for any continuous $f:G(\mathbb{H})\to\mathbb{R}$ with compact support, we have
$$
\sup_{\gamma\in Isom(\mathbb{H})}|\int_{G(\mathbb{H})}fd[\gamma^{*}(\alpha_n)]|\leq C(f)
$$
where $C(f)$ is independent of $n$. By choosing $f:G(\mathbb{H})\to\mathbb{R}$ to be positive and $f|_{Q^{*}}=1$, we get that
$\sup_{L(Q)=\log 2}\alpha_n(Q)<C(f)$ for all $n$; thus
 $h_n=\mathcal{L}^{-1}(\alpha_n)$ is 
bounded in $T(\mathbb{H})$. 

It follows that there exists a 
subsequence $h_{n_k}$ which pointwise converges to a quasisymmetric map $h$ on $S^1$. Let $\beta=\mathcal{L}(h)$. Thus
$$
\alpha_n(Q)\to\beta (Q)
$$
as $n\to\infty$ for each box of geodesics $Q$. Thus $\alpha =\beta$ by the uniqueness of measures.

Finally, $\mathcal{L}(T(\mathbb{H}))$ is clearly unbounded and $\mathcal{L}$ is a proper map because $\mathcal{L}^{-1}(M)$ is bounded whenever $M\subset\mathcal{M}(G(\mathbb{H}))$ is bounded by the proof above.
\end{proof}

\section{The fundamental lemma}

The following lemma is used when considering convergence of an earthquake path $E^{t\mu}$ as $t\to\infty$ on Thurston's boundary of $T(\mathbb{H})$ in the uniform weak* topology.

\begin{lemma}
\label{lem:fund}
Let $\beta_n\in ML_{bdd}(\mathbb{H})$ be a bounded (in Thurston's norm) sequence that converges in the weak* topology to $\beta\in ML_{bdd}(\mathbb{H})$.
Assume $Q=[a,b]\times [c,d]$ is a box of geodesics with $\beta (\partial Q)=0$. Then, for $t_n>0$ and $t_n\to\infty$ as $n\to\infty$, 
$$
\frac{1}{t_n}L(E^{t_n\beta_n}(Q))\to \beta (Q)
$$
as $n\to\infty$, where $E^{t_n\beta_n}$ is an earthquake  path with an earthquake measure $t_n\beta_n$.
\end{lemma}

\begin{remark}
The above convergence is assumed to be in the weak* topology. However, we consider convergence of earthquakes  $E^{t_n\beta_n}$ with variable measures $t_n\beta_n$ which allows us to use this lemma when proving convergence in the uniform weak* topology in the next section.
\end{remark}

\begin{remark}
Let $E^{t_n\beta_n}|_{S^1}=h_n$. Then $h_n$ is a quasisymmetric map of $S^1$ and $h_n^{*}L$ is a full support measure on $G(\mathbb{H})$. The limiting measure $\beta$ is supported on a geodesic lamination, hence its support is small inside $G(\mathbb{H})$. 
\end{remark}

\begin{proof}
Since $\beta_n\to\beta$ in the weak* topology as $n\to\infty$ and $\beta (\partial Q)=0$ we have  $\beta_n(Q)\to\beta (Q)$ as $n\to\infty$. 

We first give an upper bound to $\lim_{n\to\infty}\frac{1}{t_n}L(E^{t_n\beta_n}(Q))$.
Fix $\epsilon >0$.
Let $a'\in (d,a)$ and $c'\in (b,c)$ be such that $$\beta ([a',a]\times [c,d]), \beta ([a,b]\times [c',c])<\frac{\epsilon}{2}.$$ 
Since a positive, countably additive, finite measure can have at most countably many disjoint sets of non-zero measure, it follows that $a'$ and $c'$ can be chosen such that $$\beta (\partial ([a',a]\times [c,d]))=\beta (\partial ([a,b]\times [c',c]))=\beta (\partial ([a',a]\times [c',c]))=0.$$ 
Then there exists $n_0=n_0(\epsilon )$ such that, for all $n\geq n_0$, $$\beta_n ([a',a]\times [c,d])<\epsilon, $$ 
$$\beta_n ([a,b]\times [c',c])<\epsilon $$
and
$$\beta_n ([a',a]\times [c',c])<\epsilon .$$

We partition measured lamination $\beta$ into a finite sum of measured laminations as follows (cf. Figure 1)
\begin{equation}
\label{eq:division}
\begin{array}l
\beta_1^n(B)=\beta_n(B\cap Q),\\
\beta_2^n(B)=\beta_n(B\cap [a,b]\times [c',c)),\\
\beta_3^n(B)=\beta_n(B\cap [a,b]\times [b,c')),\\
\beta_4^n(B)=\beta_n(B\cap [a',a)\times [c,d]),\\
\beta_5^n(B)=\beta_n(B\cap [d,a')\times [c,d]),\\
\beta_6^n(B)=\beta_n(B)-\sum_{i=1}^5\beta_i^n(B),
\end{array}
\end{equation}
where $B\subset G(\mathbb{H})$ is any Borel set. 
Note that $\beta_i^n$ are defined by restricting $\beta_n$ to boxes of geodesics with some of the boxes not being closed. This is done to avoid ambiguity because an intersection of two boxes along their boundaries might have non-zero $\beta_n$-mass. For example, $\beta_2^n$ is defined by restricting to box $[a,b]\times [c',c)$ because $\beta_n([a,b]\times \{ c\})$ might be non-zero and we defined $\beta_n^1$ by restricting to box $[a,b]\times [c,d]$. In this case the support of $\beta_2^n$ might contain geodesics in $[a,b]\times\{ c\}$ while $\beta_2^n([a,b]\times\{ c\})=0$ (because the support is defined as the smallest closed set whose complement has zero mass). Similar property holds for other measures.
We divide our considerations into several cases.

\vskip .2 cm

\begin{figure}
\noindent\makebox[\textwidth]{%
\includegraphics[width=10cm]{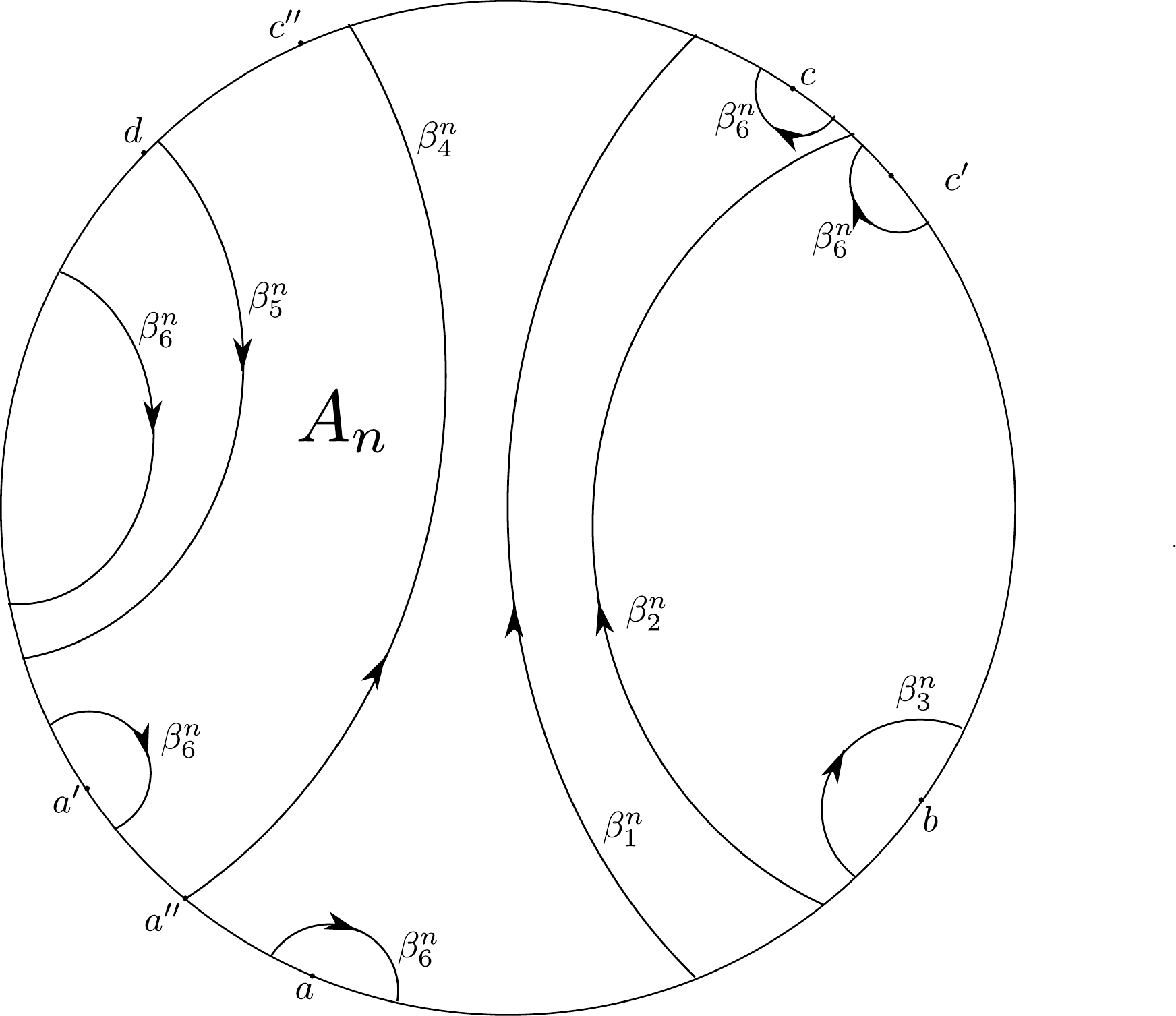}
}
\caption{An upper bound on the Liouville measure: $\beta^n_1\neq 0$.}
\end{figure}%

\vskip .2 cm

\noindent {\bf Case 1.} Assume that $\beta^n_1$ is non-trivial.  
Let $A_n$ be the stratum of $\beta_n$ that separates geodesics of $|\beta_n|$ in $[a',b]\times [c,d]$ from interval $(d,a')$ on $S^1$. In particular, if $\beta_4^n$ and $\beta_5^n$ are non-trivial then $A_n$ separates the
support of $\beta_4^n$ from the support of $\beta_5^n$.  
Note that $A_n$ could be either a hyperbolic polygon or a geodesic. In the case that $A_n$ is a geodesic then it is in the support of both $\beta_4^n$ and $\beta_5^n$. We normalize earthquakes $E^{t_n\beta_n}$ and $E^{t_n\beta_i^n}$, for $i=1,\ldots ,6$, to be the identity on a stratum (that contains) $A_n$. Let $a''$ be a point on the boundary of $A_n$ in the interval
$[a',a]$ and let $c''$ be a point of $A_n$ in the interval $[c,d]$.

Then we have
$$
E^{t_n\beta_n}|_{[a'',c'']}=E^{t_n\beta_4^n}\circ E^{t_n\beta_1^n}\circ E^{t_n\beta_2^n}\circ E^{t_n\beta_3^n}\circ E^{t_n\beta_6^n}
$$
and
$$
E^{t_n\beta_n}|_{[c'',a'']}=E^{t_n\beta_5^n}\circ E^{t_n\beta_6^n}.
$$

We estimate $L(E^{t_n\beta_n}([a,b]\times [c,d]))$ from the above. The action of earthquake $E^{t_n\beta_6^n}$ fixes points $b$ and $d$, and possibly moves $a$ towards $b$ and possibly moves $c$ towards $d$ because it moves all points to the left relative the stratum $A_n$. This decreases the Liouville measure of the box $[a,b]\times [c,d]$ and we delete $E^{t_n\beta_6^n}$ from the definition of $E^{t_n\beta_n}$.

Earthquake $E^{t_n\beta_3^n}$ moves $b$ towards $c$ and it can at most reach point $c'$. Similar, earthquake $E^{t_n\beta_5^n}$ moves $d$ towards $a$ and the closest it can get is $a'$.
Therefore, it is enough to consider the action of $E^{t_n\beta_4^n}\circ E^{t_n\beta_1^n}\circ E^{t_n\beta_2^n}$ on box $[a,c']\times [c,a']$. 

The support of $\beta_1^n+\beta_2^n+\beta_4^n$ is in $[a',b]\times [c',d]\subset [a',c']\times [c',a']$. The second inequality in Proposition \ref{prop:min-max-earthquake} implies that 
$$
L(E^{t_n\beta_n}([a,b]\times [c,d]))\leq L([a,T_n(c')]\times [c,a'])
$$
where $T_n$ is a hyperbolic translation with repelling fixed point $a$, attracting fixed point $c$ and translation length $t_n(\beta_1^n+\beta_2^n+\beta_4^n)$. Then Lemma \ref{lem:estimate_L_box_Epath} gives
\begin{equation}
\label{eq:case1}
\begin{split}
L(E^{t_n\beta_n}([a,b]\times [c,d]))\leq t_n(\beta_1^n+\beta_2^n+\beta_4^n)([a',b]\times [c',d])+L([a,c']\times [c,a'])\\ \leq t_n[\beta_n([a,b]\times [c,d])+4\epsilon]
+L([a,c']\times [c,a']).
\end{split}
\end{equation}

 \vskip .2 cm

\noindent {\bf Case 2.} Assume that $\beta^n_1$ is trivial and that either $\beta^n_4$ or $\beta^n_2$ is non-trivial. Let $A_n$ be a stratum of $\beta_n$ that separates the support $|\beta^n_4|\cup |\beta^n_2|$ from $[d,a']$. The reasoning in Case 1 applies in this case as well and we obtain
\begin{equation}
\label{eq:case2}
L(E^{t_n\beta_n}([a,b]\times [c,d]))\leq t_n 4\epsilon
+L([a,c']\times [c,a']).
\end{equation}

\vskip .2 cm

\noindent {\bf Case 3.} Assume that $\beta^n_1+\beta^n_2+\beta^n_4$ from (\ref{eq:division}) is trivial. If $\beta_n([d,a]\times [b,c])=0$ then the reasoning in the above case applies to get
$$
L(E^{t_n\beta_n}([a,b]\times [c,d]))\leq t_n 4\epsilon
+L([a,c']\times [c,a']).
$$
Assume that  $\beta_n([d,a]\times [b,c])\neq 0$. We introduce a new division of $\beta_n$ as follows. For a Borel set $B\subset G(\mathbb{H})$ we define (cf. Figure 2)
\begin{equation}
\label{eq:division1}
\begin{array}l
\gamma_1^n(B)=\beta_n(B\cap [b,c']\times [d,a']),\\
\gamma_2^n(B)=\beta_n(B\cap (c',c]\times (a',a]),\\
\gamma_3^n(B)=\beta_n(B\cap ([a,b]\times [b,c]\cup [c,d]\times [d,a])),\\
\gamma_4^n(B)=\beta_n(B\cap ([d,a]\times [a,b]\cup [b,c]\times [c,d])),\\
\gamma_5^n(B)=\beta_n(B)-\sum_{i=1}^4\gamma_i^n(B),
\end{array}
\end{equation}
Note that either $\gamma^n_1$ or $\gamma^n_2$ is trivial. We normalize $E^{t_n\beta_n}$ and $E^{t_n\gamma^n_i}$ for $i=1,2,\ldots ,5$ to be the identity on a stratum (that contains a stratum) $A_n$ of $\beta_n$ that separates $|\gamma^n_1|\cup |\gamma^n_2|$ from interval $[c,d]$ on $S^1$. Then
$$
E^{t_n\beta_n}=E^{t_n\gamma^n_1}\circ E^{t_n\gamma^n_2}\circ E^{t_n\gamma^n_3}\circ E^{t_n\gamma^n_4}\circ E^{t_n\gamma^n_5}.
$$

\begin{figure}
\noindent\makebox[\textwidth]{
\includegraphics[width=10cm]{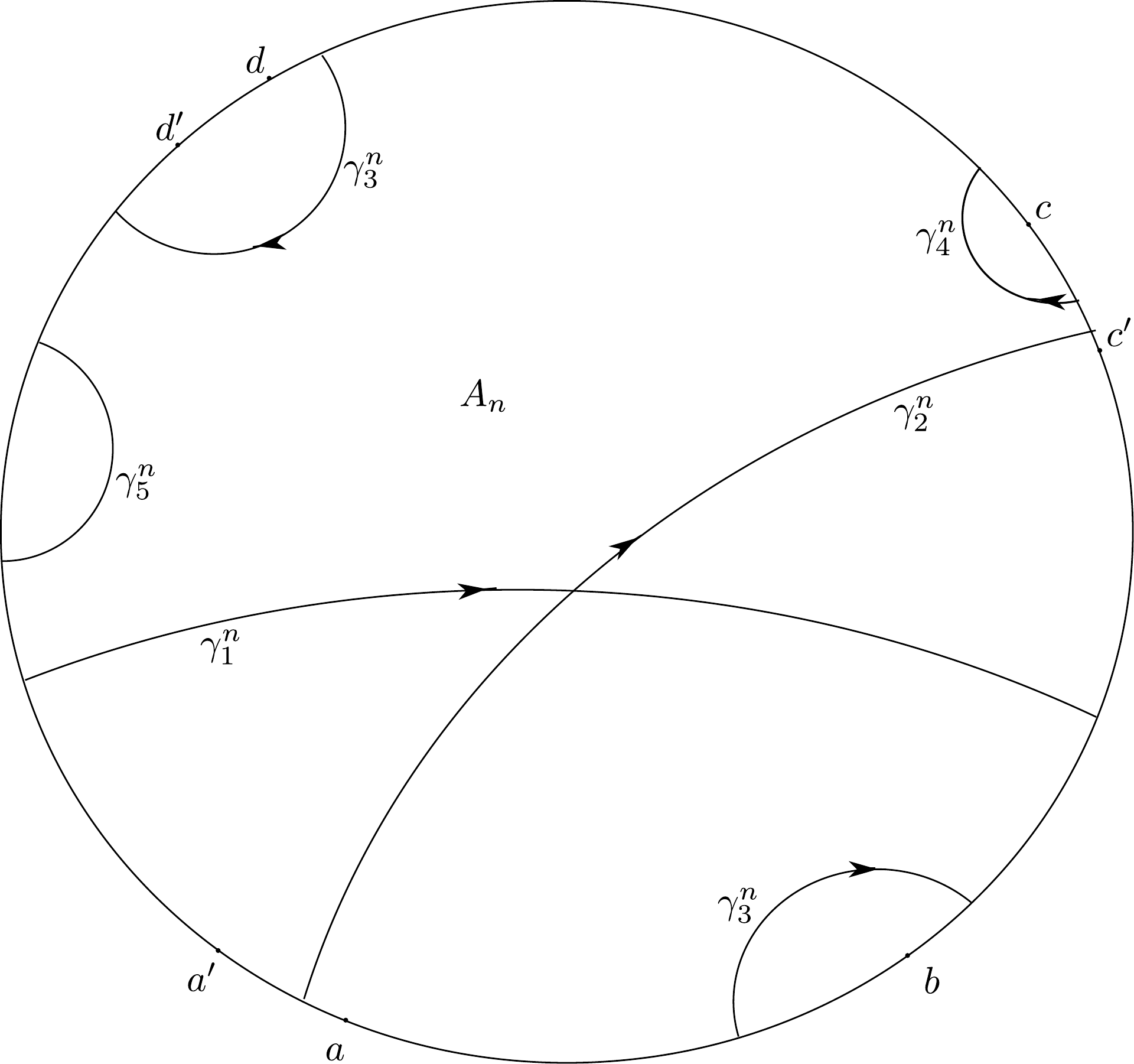}}
\caption{An upper bound on the Liouville measure: $\beta^n_1+\beta^n_2+\beta^n_4= 0$. Either $\gamma^n_1$ or $\gamma^n_2$ is trivial.}
\end{figure}

Note that $E^{t_n\gamma^n_5}$ fixes $a$, $b$, $c$ and $d$, and we can ignore it. Moreover, $E^{t_n\gamma^n_4}$ moves point $a$ towards $b$, and it moves $c$ towards $d$, and it can only decrease the Liouville measure of $[a,b]\times [c,d]$. Therefore we can ignore $E^{t_n\gamma^n_4}$. In addition, $E^{t_n\gamma^n_3}$ can move $b$ counterclockwise to at most $c'$, and it can move $d$ counterclockwise to at most $a'$, and it fixes $a$ and $c$. Therefore it is enough to consider the action of $E^{t_n\gamma^n_1}\circ E^{t_n\gamma^n_2}$ on $[a,c']\times [c,a']$. 

Assume first that $\gamma^n_1$ is trivial. Then the above and Proposition \ref{prop:min-max-earthquake} give 
\begin{equation}
\label{eq:case3.1}
\begin{split}
L(E^{t_n\beta_n})([a,b]\times [c,d]))\leq L(E^{t_n\gamma^n_2})([a,c']\times [c,a']))
\leq  t_n\beta_n(\\ [a',a]\times  [c',c])+L([a,c']\times [c,a'])\leq
t_n\epsilon +L([a,c']\times [c,a'])
\end{split}
\end{equation}

Assume next that $\gamma^n_2$ is trivial. Then $E^{t_n\gamma^n_1}$ fixes $c'$, it moves $a$ counterclockwise towards $c'$ and it fixes $c$ and $d$. Therefore
\begin{equation}
\label{eq:case3.2}
L(E^{t_n\beta_n}([a,b]\times [c,d]))\leq L([a,c']\times [c,a']).
\end{equation}

\vskip .1 cm

By dividing equations (\ref{eq:case1}), (\ref{eq:case2}), (\ref{eq:case3.1}) and (\ref{eq:case3.2}) with $t_n$ and letting $n\to\infty$ we get
$$
\limsup_{n\to\infty}\frac{1}{t_n}L(E^{t_n\beta_n}([a,b]\times [c,d])\leq \beta ([a,b]\times [c,d])
$$
since $\epsilon >0$ was arbitrary.

\vskip .2 cm
 
\noindent {\bf A lower bound.} We a find {\it lower} bound for $L(E^{t_n\beta_n}([a,b]\times [c,d]))$ when $n$ is large enough. Let $\epsilon >0$ be fixed. Since $\beta (\partial Q)=0$, it follows that there exists $b'\in (a,b)$ and $d'\in (c,d)$ such that (cf. Figure 3)
$$
\beta ([b',b]\times [c,d])+\beta ([a,b]\times [d',d])\leq \epsilon /2.
$$

\begin{figure}
\noindent\makebox[\textwidth]{
\includegraphics[width=10cm]{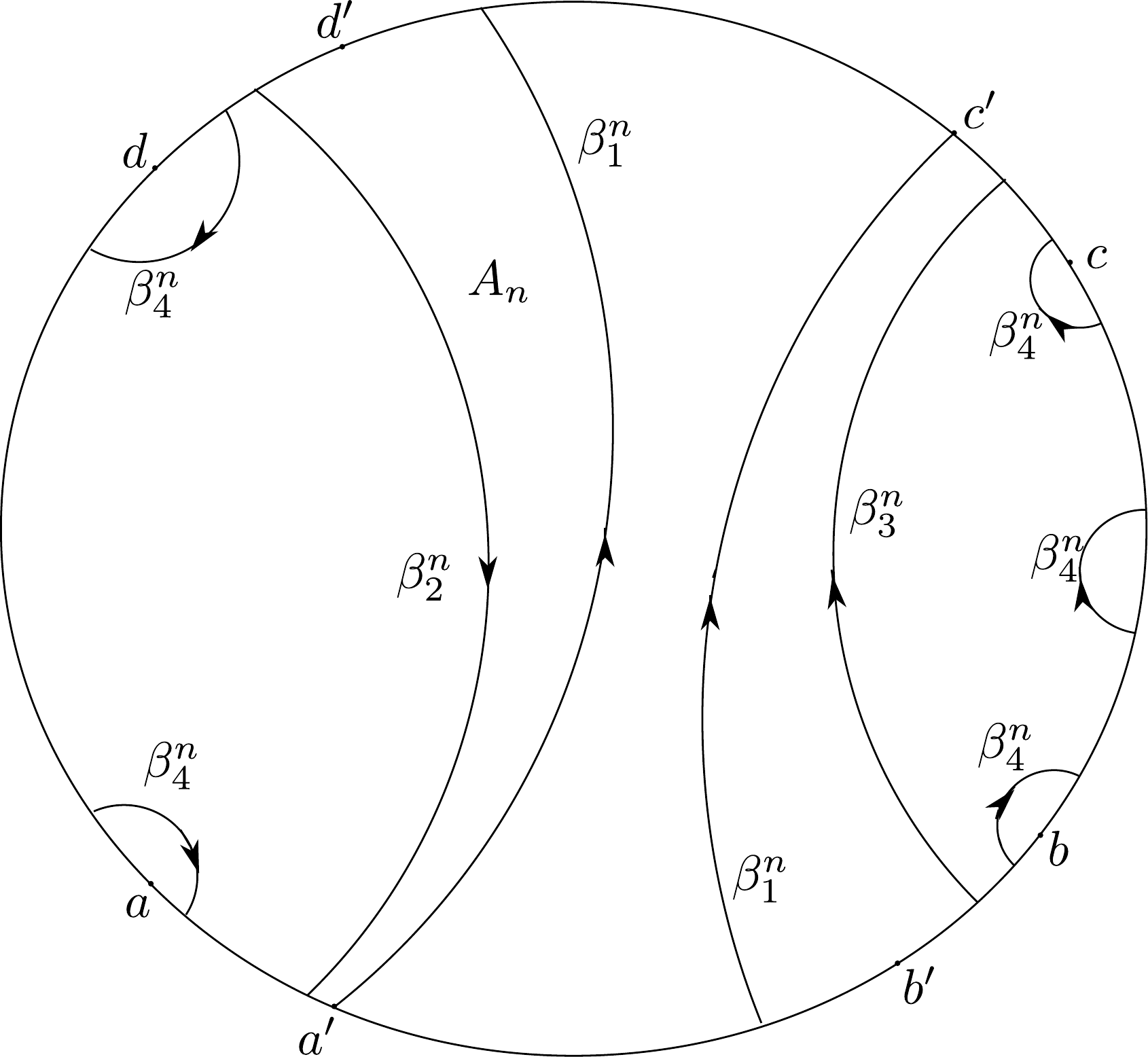}}
\caption{A lower bound on the Liouville measure.}
\end{figure}

In addition to satisfying above inequality, we can choose $b'$ and $d'$ such that
$$
\beta (\{ b'\}\times [c,d])+\beta ([a,b]\times \{ d'\} )=0 
$$
because a positive, countably additive, finite measure can have at most countably many disjoint sets of non-zero measure while we have uncountably many choices of $b'$ and $d'$. Then
$$
\beta (\partial ([b',b]\times [c,d]))+\beta (\partial ([a,b]\times [d',d]))=0
$$
which implies 
$$
\beta_n([b',b]\times [c,d])\to \beta ([b',b]\times [c,d])
$$
and 
$$
\beta_n([a,b]\times [d',d])\to \beta ([a,b]\times [d',d])
$$
as $n\to\infty$. This implies
$$
\beta_n ([b',b]\times [c,d])+\beta_n ([a,b]\times [d',d])\leq \epsilon .
$$
for all $n\geq n_0(\epsilon )$.

Let $c'\in [c,d']$ be an endpoint of a geodesic in $|\beta_n|\cap ([a,b']\times [c,d'])$ that is closest to $c$ in interval $[c,d']$, and $c'=d'$ if $|\beta_n|\cap ([a,b']\times [c,d'])=\emptyset$. 
Let $a'\in [a,b']$ be an endpoint of a geodesic in $|\beta_n|\cap ([a,b']\times [c,d'])$ that is closest to $a$ in the interval $[a,b']$, and $a'=b'$ if $|\beta_n|\cap ([a,b']\times [c,d'])=\emptyset$ (cf. Figure 3). 

We fix $n\geq n_0(\epsilon )$ and
 write $\beta_n$ as a finite sum of measured laminations as follows. For a Borel set $B\subset G(\mathbb{H})$, define
\begin{equation}
\begin{array}l
\beta_1^n(B)=\beta_n(B\cap([a',b']\times [c',d'])),\\
\beta_2^n(B)=\beta_n(B\cap ([a,b]\times (d',d]))),\\
\beta_3^n(B)=\beta_n(B\cap ((b',b]\times [c,d]))),\\
\beta_4^n(B)=\beta_n(B)-\sum_{i=1}^3\beta_i^n(B).
\end{array}
\end{equation}
We divide the analysis into several cases.

\vskip .2 cm

\noindent {\bf Case 1.} Assume that $\beta_1^n$ is non-trivial. This implies that no geodesic of the support of $\beta^n_4$ is in $[d,a]\times [b,c]$.
Normalize earthquakes $E^{t_n\beta_n}$ and $E^{t_n\beta_i^n}$, for $i=1,2,3,4$, to be the identity on a stratum (that contains stratum) $A_n$ of $\beta_n$ that separates the support $|\beta_1^n|$ of $\beta^n_1$ from $[d,a]\subset S^1$. Note that the stratum $A_n$ might be a  geodesic. We have
\begin{equation}
\begin{array}l
E^{t_n\beta_n}|_{[a',d']}=E^{t_n\beta_1^n}\circ E^{t_n\beta_3^n}\circ E^{t_n\beta_4^n},\\
E^{t_n\beta_n}|_{[d',a']}=E^{t_n\beta_2^n}\circ E^{t_n\beta_4^n}.
\end{array}
\end{equation}

We consider the image of $[a',b]\times [c',d]$ under $E^{t_n\beta_n}$. Since $E^{t_n\beta_4^n}$ is a left earthquake, chosen normalization implies that $a'$ 
and $c'$ are fixed, and possibly $b$ is moved towards $c'$, and possibly  $d$ is moved towards $a'$ for the fixed orientation on $S^1$. These movements  increase Liouville measure and since we are looking for a lower bound, we ignore the action of $E^{t_n\beta_4^n}$. In a similar fashion, earthquakes $E^{t_n\beta_2^n}$ and $E^{t_n\beta_3^n}$ can only increase Liouville measure of $[a',b]\times [c',d]$ and we ignore them.

It remains to estimate Liouville measure of $E^{t_n\beta_1^n}([a',b]\times [c',d])$. By Proposition \ref{prop:min-max-earthquake}, we have that $L(E^{t_n\beta_1^n}([a',b]\times [c',d]))$ is larger than $L([a',T(b)]\times [T(c'),d])$, where $T$ is a hyperbolic translation with translation length $t_n\beta_1^n([a',b]\times [c',d])$ whose repelling fixed point is $b'$ and attracting fixed point is $d'$. 
 
From above we obtain
$$
L(E^{t_n\beta_n}([a,b]\times [c,d]))\geq L(E^{t_n\beta_1^n}([a',b]\times [c',d]))\geq L([b',T(b)]\times [d',d])
$$
and Lemma \ref{lem:estimate_L_box_Epath} gives
$$
L([b',T(b)]\times [d',d])\geq t_n\beta_1^n([a',b]\times [c',d])+\log\frac{D^2}{4}
$$
where $D$ is the distance between geodesics $l(b',d)$ and $l(b,d')$. The above choice of $b'$, $d'$ and  $\epsilon >0$ gives, for all $n\geq n_0(\epsilon )$,
\begin{equation}
\label{eq:l-ineq1}
L(E^{t_n\beta_n}([a,b]\times [c,d]))\geq t_n(\beta_n ([a,b]\times [c,d])-4\epsilon) +\log\frac{D^2}{4}.
\end{equation}

\vskip .2 cm

\noindent {\bf Case 2.} Assume that $\beta^n_1$ is trivial and that either $\beta^n_2$ or $\beta^n_3$ is non-trivial. In this case no geodesic of the support of $\beta^n_4$ belongs to $[d,a]\times [b,c]$. We normalize the earthquakes $E^{t_n\beta_n}$ and $E^{t_n\beta^n_i}$ for $i=1,2,3,4$ as in the previous case. Note that $E^{t_n\beta^n_1}=id$. As in the previous case, all earthquakes $E^{t_n\beta^n_i}$ for $i=2,3,4$ can only increase the Liouville measure of $[a',b]\times [c',d]$. Then we have
\begin{equation}
\label{eq:l-ineq2}
L(E^{t_n\beta_n}([a,b]\times [c,d]))\geq \log\frac{D^2}{4}\geq
t_n(\beta_n([a,b]\times [c,d])-4\epsilon)+\log\frac{D^2}{4},
\end{equation} 
since $\beta^n_1([a,b]\times [c,d])=0$ and $\beta_n([a,b]\times [c,d])\leq 4\epsilon$, where $D$ is the distance between geodesics $l(b',d)$ and $l(b,d')$.

\vskip .2 cm

\noindent {\bf Case 3.} Assume that $\beta^n_i$ for $i=1,2,3$ are trivial. Then $\beta_n([a,b]\times [c,d])=0$ and
\begin{equation}
\label{eq:l-ineq3}
L(E^{t_n\beta_n}([a,b]\times [c,d]))\geq 0=
t_n(\beta_n([a,b]\times [c,d]).
\end{equation} 

\vskip .2 cm

By dividing each inequality (\ref{eq:l-ineq1}), (\ref{eq:l-ineq2}) and (\ref{eq:l-ineq3}) with $t_n$ and letting $n\to\infty$ together with the fact that $\epsilon$ was arbitrary, we get
$$
\liminf_{n\to\infty}\frac{1}{t_n}L(E^{t_n\beta_n}([a,b]\times [c,d]))\geq\beta ([a,b]\times [c,d])).
$$

\end{proof}

\section{Convergence of earthquake paths in Thurston's closure}

We first prove that each box of geodesics $Q=[a,b]\times [c,d]$ is the limit (in the Hausdorf topology) of a sequence of increasing (in the sense of inclusions) boxes $Q_n$ with $\beta (\partial Q_n)=0$. Indeed, $\partial Q=(\{ a\}\times [c,d])\cup (\{b\}\times [c,d])\cup ([a,b]\times \{ c\})\cup( [a,b]\times\{ d\})$. 
Consider a small open interval $I_a\subset S^1$ around $a$. Since $\beta$ is locally finite, 
there exists at most countably many $a'\in I_a$ such that $\beta (\{ a'\}\times [c,d])>0$. Choose $a_n\in I_a\cap (a,d]$ such that $\beta (\{ a_n\}\times [c,d])=0$. Similarly we choose $b_n$ close to $b$ such that $\beta (\{ b_n\}\times [c,d])=0$. In the same fashion, we choose $c_n$ close to $c$ and $d_n$ close to $d$ such that 
$$
\beta (\partial ([a_n,b_n]\times [c_n,d_n]))=0
$$
and set $Q_n=[a_n,b_n]\times [c_n,d_n]$.

Next we prove the convergence of the earthquake paths in Thurston's boundary which establish Theorem 2 in Introduction.

\begin{theorem}
\label{thm:earthquake_path_conv}
Let $\beta\in ML_{bdd}(\mathbb{H})$ and let $E^{t\beta}$, for $t>0$, be a left earthquake with an earthquake measure $t\beta$. Then
$$
\frac{1}{t}(E^{t\beta}|_{S^1})^{*}L\to\beta
$$
as $t\to\infty$ in the uniform weak* topology on $\mathcal{M}(G(\mathbb{H}))$.
\end{theorem}

\begin{proof}
Without loss of generality we can assume that $\|\beta\|_{Th}=1$. Let $h_t=E^{t\beta}|_{S^1}$, for $t>0$, be the restriction of earthquake path $E^{t\beta}$ to the boundary $S^1$ of $\mathbb{H}$. Let
$$
\alpha_t=(h_t)^{*}L
$$
be the image of $h_t\in T(\mathbb{H})$ in $\mathcal{M}(G(\mathbb{H}))$. 

Assume on the contrary that $\frac{1}{t}\alpha_t$ does not converge to $\beta$ in the uniform weak* topology as $t\to\infty$. 
Then there exists a continuous function $f:G(\mathbb{H})\to\mathbb{R}$ with compact support, a sequence of isometries $\gamma_n\in Mob(\mathbb{H})$, and a sequence $t_n\to \infty$ as $n\to\infty$ such that, for all $n\in\mathbb{N}$,
\begin{equation}
\label{eq:contr_nonconv}
\Big{|}\int_{G(\mathbb{H})}f d\Big{[}\Big{(}\gamma_{n}\Big{)}^{*}\Big{(}\frac{1}{t_n}\alpha_{t_n}-\beta \Big{)}\Big{]}\Big{|}\geq C_0>0.
\end{equation}

Define
$$
\alpha_{t_n}'=(\gamma_{n})^{*}\alpha_{t_n}
$$
and
$$
\beta_n=(\gamma_{n})^{*}\beta .
$$

Let $Q=[a,b]\times [c,d]$ be an arbitrary box of geodesics.
By Lemma \ref{lem:estimate_L_box_Epath}
\begin{equation*}
\frac{1}{t_n}\alpha_{t_n}'(Q)\leq \beta_n (Q)+\frac{1}{t_n}L(Q)\leq \Big{(}\frac{L(Q)}{\log 2}+1\Big{)}
\|\beta\|_{Th}+\frac{1}{t_n}L(Q)= C(Q)
\end{equation*}
for all $n$ such that $t_n\geq 1$. Also 
\begin{equation*}
\beta_n(Q)\leq \Big{(}\frac{L(Q)}{\log 2}+1\Big{)}\|\beta\|_{Th}
\end{equation*}
for all $n$.

The above two inequalities imply that both $\beta_n$ and $\frac{1}{t_n}\alpha_{t_n}'$ are 
uniformly bounded on each box $Q$. Then there exist subsequences $\frac{1}{t_{n_k}}
\alpha_{t_{n_k}}'$ and $\beta_{n_k}$ that converge in the weak* topology on $\mathcal{M}(G(\mathbb{H}))$ to $\alpha^{\#}$ and $
\beta^{\#}$, respectively, as $k\to\infty$.

Then (\ref{eq:contr_nonconv}) gives
\begin{equation}
\label{eq:not_equal_m1}
\Big{|}\int_{G(\mathbb{H})}fd(\alpha^{\#}-\beta^{\#})\Big{|}\geq C_0.
\end{equation}

On the other hand, Lemma \ref{lem:fund} implies that $\alpha^{\#}$ and $\beta^{\#}$ agree on all boxes $Q^{\#}$ with $\beta^{\#}(\partial Q^{\#})=0$. These boxes are dense among all boxes in $G(\mathbb{H})$ and $\alpha^{\#}=\beta^{\#}$ contradicting (\ref{eq:not_equal_m1}). The contradiction proves theorem.
\end{proof}

\vskip .2 cm

The above theorem proves that Thurston's boundary contains the space of  projective bounded measured laminations. It remains to prove the opposite.

\vskip .2 cm

\begin{proposition}
A limit point of $P(\mathcal{L}(T(\mathbb{H})))$ in $P\mathcal{M}(G(\mathbb{H}))$ is necessarily a projective bounded measured lamination. 
\end{proposition}

\begin{proof}
Let $\beta$ be the limit point of a sequence $[\alpha_k]\in P(\mathcal{L}(T(\mathbb{H})))$,
where $[\alpha_k]$ is the projective class of $\alpha_k\in \mathcal{L}(T(\mathbb{H}))$.  
Then there exists $t_k\to \infty$ as $k\to \infty$ such that
$$
\frac{1}{t_k}\alpha_k\to \beta
$$
as $k\to\infty$ in the uniform weak* topology. 

Recall that $\alpha\in \mathcal{L}(T(\mathbb{H}))$ implies that 
$$
e^{-\alpha ([a,b]\times [c,d])}+e^{-\alpha ([b,c]\times [d,a])}=1
$$
for all boxes $[a,b]\times [c,d]$ of $G(\mathbb{H})$ (cf. Bonahon \cite{Bo}). 
This implies that if $\alpha_k ([a,b]\times [c,d])\to\infty$ then $\alpha_k 
([b,c]\times [d,a])\to 0$ as $k\to\infty$. 

Assume that the support of $\beta$ contains two intersecting geodesics $(m,n)\in G(\mathbb{H})$ and $(p,q)\in G(\mathbb{H})$. The geodesic $(m,n)\in G(\mathbb{H})$ separates $p$ and $q$. 
There exists a box of geodesics $Q_{(m,n)}=[a_1,b_1]\times [c_1,d_1]$ containing $(m,n)$ and a box of geodesics 
$Q_{(p,q)}=[a_2,b_2]\times [c_2,d_2]$ containing $(p,q)$ such that $[a_1,b_1]\subset (b_2,c_2)$ and $[c_1,d_1]\subset (d_2,a_2)$. Namely, every geodesic of $Q_{(m,n)}$ intersects every geodesic of $Q_{(p,q)}$.

Since $t_k\to\infty$ as $k\to\infty$ and both geodesics $(m,n)$ and $(p,q)$ are in the support of $\beta$, we have that $\alpha_k(Q_{(m,n)})\to\infty$ and $\alpha_k(Q_{(p,q)})\to\infty$ as $k\to\infty$. The boxes are chosen such that $Q_{(m,n)}\subset [b_2,c_2]\times [d_2,a_2]$. This implies $\alpha_k([b_2,c_2]\times [d_2,a_2])\to\infty$ as $k\to\infty$ which is in a contradiction with $\alpha_k(Q_{(p,q)})\to\infty$. Thus the geodesics of the support of $\beta$ do not intersect. Therefore $\beta$ is a measured lamination. Boundedness of $\beta$ follows because $\mathcal{L}(G(\mathbb{H}))$ consists of bounded measures.
\end{proof}

The proof of Theorem 1 and 2 from Introduction is now completed.

\section{Quasiconformal Mapping Class group}

The quasiconformal mapping class group $MCG_{qc}(X)$ of a hyperbolic surface $X$ consists of all quasiconformal maps $g:X\to X$ modulo homotopies bounded in the hyperbolic geometry (cf. \cite{GL}). The natural action of $MCG_{qc}(X)$ onto $T(X)$ given by $[f]\mapsto [f\circ g^{-1}]$ is continuous (cf. \cite{GL}). 

Let $X=\mathbb{H}/\Gamma$, where $\Gamma$ is a Fuchsian group. Then the Teichm\"uller space $T(X)$ is identified with the space of all quasisymmetric maps of $S^1$ that fix $1$, $i$ and $-1$, and that conjugate $\Gamma$ to another Fuchsian group. The quasiconformal mapping class group $MCG_{qc}(X)$ is identified with the group of quasisymmetric maps of $S^1$ that conjugate $\Gamma$ onto itself. 

For the universal Teichm\"uller space 
\begin{equation*}
T(\mathbb{H})=\{ h:S^1\to S^1|h \mbox{ is quasisymmetric and fixes }1,\ i,\mbox{ and }-1\}, 
\end{equation*}
the mapping class group is given by
 \begin{equation*}
 MCG_{qc}(\mathbb{H})=\{ g:S^1\to S^1|g \mbox{ is quasisymmetric}\}.
\end{equation*}
The action of $g\in MCG_{qc}(\mathbb{H})$ is given by
$$
h\mapsto \gamma\circ h\circ g^{-1},
$$
where $\gamma\in Mob(S^1)$ such that $\gamma\circ h\circ g^{-1}$ fixes $1$, $i$ and $-1$. 
We prove that the action extends continuously to $T(\mathbb{H})\cup PML_{bdd}(\mathbb{H})$ which is Theorem 3 in Introduction. As before, the proof for the universal Teichm\"uller space extends to all Teichm\"uller spaces by the invariance under Fuchsian groups. 

\begin{theorem}
\label{thm:extension_of_mcg_action}
The action of $MCG_{qc}(\mathbb{H})$ on $T(\mathbb{H})$ extends to a continuous action on Thurston's closure $T(\mathbb{H})\cup PML_{bdd}(\mathbb{H})$.
\end{theorem}

\begin{proof}
Assume that $h_n\to [\beta ]\in PML_{bdd}(\mathbb{H})$. Namely, if $\alpha_n=(h_n)^{*}(L)$ then for any continuous $f:G(\mathbb{H})\to\mathbb{R}$ with compact support and some $\beta_1\in [\beta ]$, 
$$
\sup_{\gamma\in Mob(S^1)}\Big{|}\int_{G(\mathbb{H})} fd\gamma^{*}(\alpha_n-\beta_1 )\Big{|}\to 0
$$
as $n\to\infty$. 

The action $h_n\mapsto \gamma_n\circ h_n\circ g^{-1}$
for appropriate $\gamma_n\in Mob(S^1)$, 
gives 
$$
g^{*}(\alpha_n)=\alpha_n':=(\gamma_n\circ h_n\circ g^{-1})^{*}(L)=\alpha_n\circ g^{-1}
$$
and 
$$g^{*}(\beta_1 ):=\beta_1\circ g^{-1}.
$$
We have
\begin{equation*}
\begin{array}l
\int_{G(\mathbb{H})}f(x)d\alpha_n'(x)=
\int_{G(\mathbb{H})}f(x)d\alpha_n(h^{-1}(x))\\
=\int_{G(\mathbb{H})}f\circ h(h^{-1}(x))d\alpha_n(h^{-1}(x))=\int_{G(\mathbb{H})}f(y)d\alpha_n(y)
\end{array}
\end{equation*}
and then
\begin{equation*}
\begin{array}l
\sup_{\gamma\in Mob(S^1)}\Big{|}\int_{G(\mathbb{H})} f(x)d\gamma^{*}(\alpha_n'-\beta_1\circ g^{-1})(x)\Big{|}=\\
\sup_{\gamma\in Mob(S^1)}\Big{|}\int_{G(\mathbb{H})} f(y)d\gamma^{*}(\alpha_n-\beta_1)(y)\Big{|}\to 0
\end{array}
\end{equation*}
as $n\to\infty$. Thus the action of $MCG_{qc}(\mathbb{H})$ extends to a continuous function on Thurston's closure.
\end{proof}

\section{Appendix}

The results in this section are used in the proof of  Lemma \ref{lem:fund} in \S 6.
We prove a standard lemma regarding neighborhoods in $T(\mathbb{H})$ and 
Liouville measure of boxes of geodesics under the maps in given neighborhoods. 

\begin{lemma}
\label{lem:LMeas_Teich_neigh}
Let $h_0\in T(\mathbb{H})$. Given $\epsilon >0$ and $0<\delta <\log 2$, there exists an open neighborhood $N(h_0,\delta ,\epsilon )$ of $h_0$ in $T(\mathbb{H})$ such that for each box of geodesics $Q$ with
$$
\delta \leq L(Q)\leq\log 2
$$
we have
$$
|\alpha_0(Q)-\alpha (Q)|<\epsilon
$$
where $\alpha_0=(h_0)^{*}L$ and $\alpha =h^{*}L$, for any $h\in N(h_0,\delta ,\epsilon )$.
\end{lemma}

\begin{proof}
Given a box of geodesics $Q=[a,b]\times [c,d]$, let $m(Q)$ denote the modulus of the quadrilateral with interior $\mathbb{H}$ whose $a$-sides are $[a,b],[c,d]\subset S^1$ and $b$-sides are $[b,c],[d,a]\subset S^1$. Then $m(Q)$ and $L(Q)$ are continuous functions of each other with $m(Q)=1$ if and only if $L(Q)=\log 2$. 

If $f_0:\mathbb{H}\to\mathbb{H}$ is a $K$-quasiconformal continuous extension of $h_0:S^1\to S^1$ then
$$
\frac{1}{K}m(Q)\leq m(f_0(Q))\leq K m(Q)
$$
for all quadrilaterals $Q$ with interior $\mathbb{H}$. 

If $\delta \leq L(Q)\leq\log 2$ then there exists $C=C(K,\delta )\geq 1$ such that
$$
\frac{1}{C}\leq L(h_0(Q))\leq C
$$
(by the continuous dependence of $L(Q)$ on $m(Q)$).

Furthermore, there exists $C_1=C_1(C)\geq 1$ such that
$$
\frac{1}{C_1}\leq m(h_0(Q))\leq C_1
$$
for all $Q$ with $\delta\leq L(Q)\leq\log 2$.

Let $h\in T(\mathbb{H})$ such that $h\circ h_0^{-1}$ has $K_1$-quasiconformal extension to $\mathbb{H}$. Then
$$
|m(h_0(Q))-m(h(Q))|\leq (K_1-1)m(h_0(Q)).
$$
By the uniform continuity of $L(Q)$ in $m(Q)$ when $m(Q)$ is in a compact 
interval $[\frac{1}{C_1},C_1]$, we obtain
$$
|L(h_0(Q))-L(h(Q))|\to 0
$$
as $K_1\to 1$.

Since $\alpha_0(Q)=L(h_0(Q))$ and $\alpha (Q)=L(h(Q))$, there exists a 
neighborhood $N(h_0,\delta ,\epsilon )$ of $h_0\in T(\mathbb{H})$ which 
satisfies the conclusions of the lemma.
\end{proof}

We consider the behavior of the Liouville measure of a box of geodesics under a simple (left) earthquake.

\begin{lemma}
\label{lem:simple_earthq_estimate}
Let $[a,b]\times [c,d]$ be a fixed box of geodesics and let $l$ be a geodesic with endpoint $x\in [d,b]$ and $y\in [b,d]$ with $d,x,y$ in the counterclockwise order. Let $E$ be an earthquakes with support $l$ and a fixed measure $m>0$. Define
$$
f(x,y)=L(E([a,b]\times [c,d])).
$$

Then $f(x,y)$ is increasing in $x\in [d,a]$ and decreasing in $x\in [a,b]$, for a fixed $y\in [b,d]$.

Moreover, $f(x,y)$ is increasing in $y\in[b,c]$ and decreasing in $y\in [c,d]$, for a fixed $x\in [d,b]$.
\end{lemma}

\begin{proof}
Assume $x\in [d,a]$ and $y\in [b,c]$ (cf. Figure 4). Normalize such that $c<d=0\leq x\leq a<b<y=\infty$ and $a>0$. Let $T(z)=e^m(z-x)+x$ be a hyperbolic translation with repelling fixed point $x$, attracting fixed point $y=\infty$ and translation length $m>0$. Then, by definition of earthquake $E$,
$$
f(x,y)=L ([T(a),T(b)]\times [c,d]).
$$
Further, we have
$$
f(x,y)=\log\frac{[e^m(a-x)+x-c][e^m(b-x)+x]}{[e^m(a-x)+x][e^m(b-x)+x-c]}
$$
and
$$
\frac{\partial}{\partial x}f(x,y)=\frac{1-e^m}{e^m(a-x)+x-c}+\frac{1-e^m}{e^m(b-x)+x} +\frac{e^m-1}{e^m(a-x)+x}+\frac{e^m-1}{e^m(b-x)+x-c}.
$$ 
By simplifying the right side of the above equation, we get
$$
\frac{\partial}{\partial x}f(x,y)=\frac{(b-a)e^m(e^m-1)}{[e^m(b-x)+x][e^m(a-x)+x]} -\frac{(b-a)e^m(e^m-1)}{[e^m(b-x)+x-c][e^m(a-x)+x-c]}>0
$$
and $f(x,y)$ is increasing in  $x\in [d,a]$ for a fixed $y\in [b,c]$.

\begin{figure}
\noindent\makebox[\textwidth]{
\includegraphics[width=10cm]{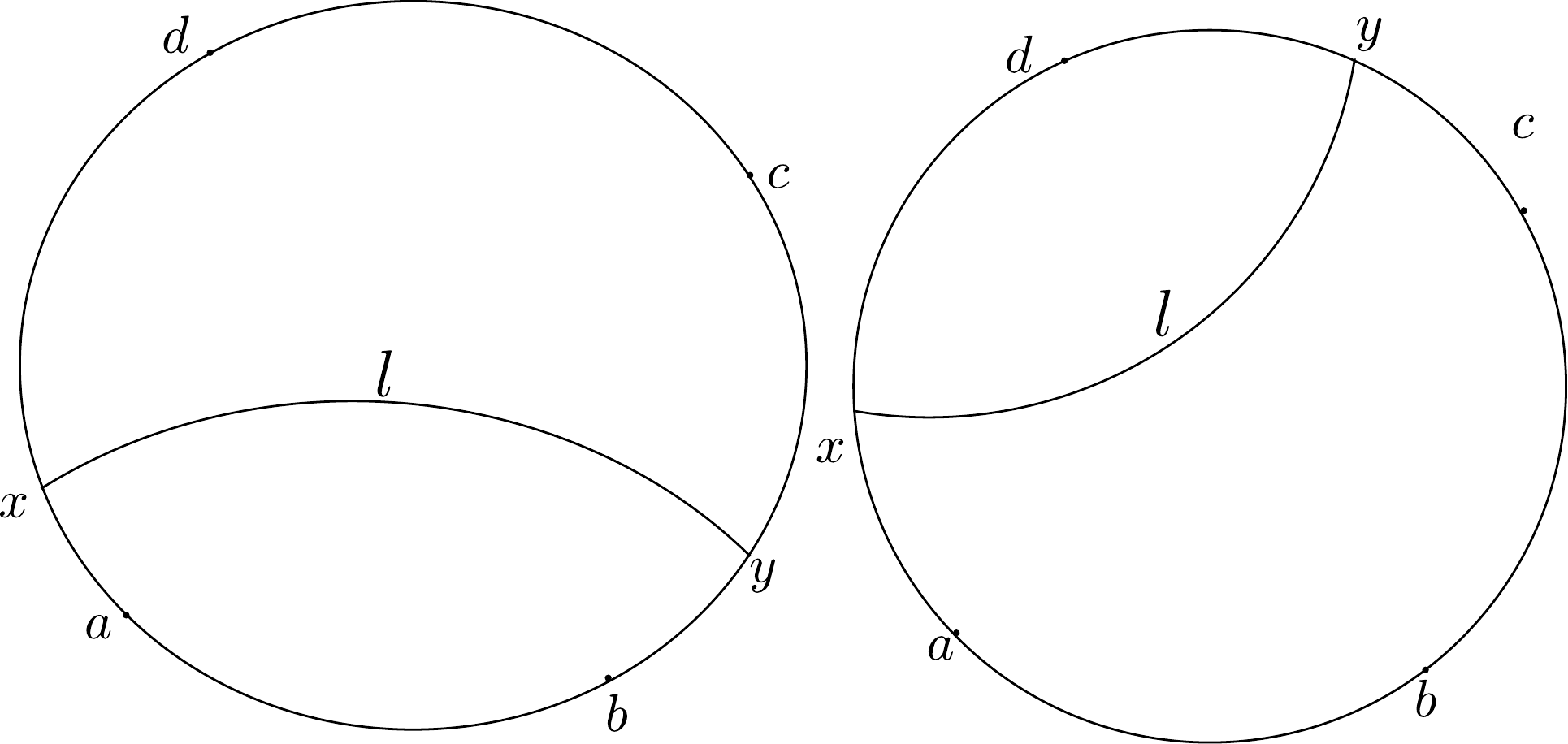}}
\caption{Estimating the Liouville measure $f(x,y)$ under a simple earthquake: $x\in [d,a]$; $y\in [b,c]$ or $y\in [c,d]$.}
\end{figure}

\vskip .2 cm

Assume $x\in [d,a]$ and $y\in [c,d]$ (cf. Figure 4). Normalize such that $d=0\leq x\leq a<b<c<y=\infty$ and $a>0$. Let $T(z)=e^{-m}(z-x)+x$. By definition of earthquake $E$, we have
$$
f(x,y)=L([a,b]\times [c, T(d)])
$$
which gives
$$
f(x,y)=\log\frac{(c-a)[b-(1-e^{-m})x]}{(c-b)[a-(1-e^{-m})x]}.
$$
Then
\begin{equation*}
\begin{split}
\frac{\partial}{\partial x}f(x,y)=\frac{-(1-e^{-m})}{b-(1-e^{-m})x}+ \frac{1-e^{-m}}{a-(1-e^{-m})x}\\ =\frac{(b-a)(1-e^{-m})}{[b-(1-e^{-m})x][a-(1-e^{-m})x]}>0
\end{split}
\end{equation*}
and $f(x,y)$ is increasing in  $x\in [d,a]$ for a fixed $y\in [c,d]$.

\vskip .2 cm

Assume $x\in [a,b]$ and $y\in [b,c]$ (cf. Figure 5). Normalize such that $c<d<a=0\leq x\leq b<y=\infty$ and $a<b$. Let $T(z)=e^{m}(z-x)+x$. By definition of earthquake $E$, we have
$$
f(x,y)=L([a,T(b)]\times [c, d])
$$
which gives
$$
f(x,y)=\log\frac{(-c)[e^{m}(b-x)+x-d]}{(-d)[e^{m}(b-x)+x-c]}.
$$
We have
\begin{equation*}
\begin{split}
\frac{\partial}{\partial x}f(x,y)=\frac{-(e^m-1)}{e^m(b-x)+x-d} +\frac{e^m-1}{e^m(b-x)+x-c}\\ = \frac{(e^m-1)(c-d)}{[e^m(b-x)+x-d][e^m(b-x)+x-c]}<0.
\end{split}
\end{equation*}

\vskip .2 cm

\begin{figure}
\noindent\makebox[\textwidth]{
\includegraphics[width=10cm]{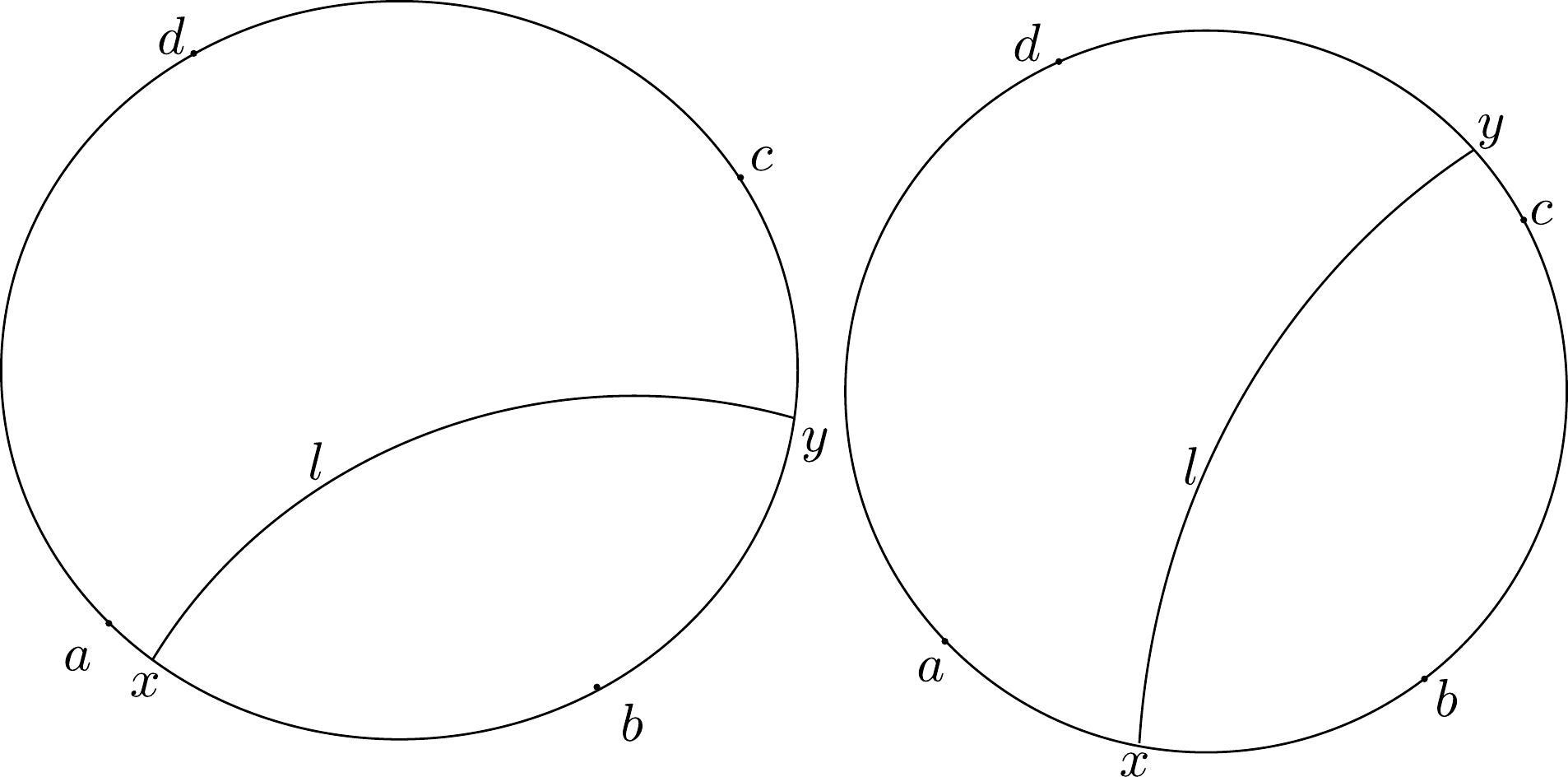}}
\caption{Estimating the Liouville measure $f(x,y)$ under a simple earthquake: $x\in [a,b]$; $y\in [b,c]$ or $y\in [c,d]$.}
\end{figure}

Assume $x\in [a,b]$ and $y\in [c,d]$ (cf. Figure 5). Normalize such that $d<a=0\leq x\leq b<c<y=\infty$ and $0<b$. Let $T(z)=e^{m}(z-x)+x$. By definition of earthquake $E$, we have
$$
f(x,y)=L([a,T(b)]\times [T(c), d])
$$
which gives
$$
f(x,y)=\log\frac{[e^m(c-x)+x][e^{m}(b-x)+x-d]}{(-d)[e^{m}(c-b)]}.
$$
We have
\begin{equation*}
\frac{\partial}{\partial x}f(x,y)=\frac{-(e^m-1)}{e^m(c-x)+x} +\frac{-(e^m-1)}{e^m(b-x)+x-d}<0.
\end{equation*}

\vskip .2 cm

Assume $y\in [b,c]$ and $x\in [d,a]$ (cf Figure 4). Normalize such that $a<b=0\leq y\leq c<d<x=\infty$ and $b<c$. Let $T(z)=e^{m}(z-y)+y$. By definition of earthquake $E$, we have
$$
f(x,y)=L([a,b]\times [T(c), T(d)])
$$
which gives
$$
f(x,y)=\log\frac{[e^m(c-y)+y-a][e^{m}(d-y)+y]}{[e^m(d-y)+y-a][e^{m}(c-y)+y]}.
$$
We have
\begin{equation*}
\begin{split}
\frac{\partial}{\partial y}f(x,y)=\frac{-(e^m-1)}{e^m(c-y)+y-a} +\frac{-(e^m-1)}{e^m(d-y)+y}+ \frac{e^m-1}{e^m(d-y)+y-a} + \frac{e^m-1}{e^m(c-y)+y} \\
=\frac{(-a)(e^m-1)}{[e^m(c-y)+y-a][e^m(c-y)+y]} +\frac{(e^m-1)a}{[e^m(d-y)+y][e^m(d-y)+y-a]}
>0.
\end{split}
\end{equation*}

\vskip .2 cm

Assume $y\in [b,c]$ and $x\in [a,b]$ (cf. Figure 5). Normalize such that $b=0\leq y\leq c<d<a<x=\infty$ and $0<c$. Let $T(z)=e^{-m}(z-y)+y$. By definition of earthquake $E$, we have
$$
f(x,y)=L([a,T(b)]\times [c, d])
$$
which gives
$$
f(x,y)=\log\frac{(a-c)[d-(1-e^{-m})y][e^{m}(d-y)+y]}{(a-d)[c-(1-e^{-m})y]}.
$$
We have
\begin{equation*}
\begin{split}
\frac{\partial}{\partial y}f(x,y)=\frac{-(1-e^{-m})}{d-(1-e^{-m})y} +\frac{1-e^{-m}}{c-(1-e^{-m})y}\\
=\frac{(d-c)(1-e^{-m})}{[d-(1-e^{-m})y][c-(1-e^{-m})y]} >0.
\end{split}
\end{equation*}

\vskip .2 cm

Assume $y\in [c,d]$ and $x\in [d,a]$ (cf. Figure 4). Normalize such that $a=0<b<c\leq y\leq d<x=\infty$ and $c<d$. Let $T(z)=e^{m}(z-y)+y$. By definition of earthquake $E$, we have
$$
f(x,y)=L([a,b]\times [c, T(d)])
$$
which gives
$$
f(x,y)=\log\frac{c[e^m(d-y)+y-b]}{(c-b)[e^m(d-y)+y]}.
$$
We have
\begin{equation*}
\begin{split}
\frac{\partial}{\partial y}f(x,y)=\frac{1-e^{m}}{e^m(d-y)+y-b} +\frac{e^{m}-1}{e^m(d-y)+y}<0.
\end{split}
\end{equation*}

\vskip .2 cm

Assume $y\in [c,d]$ and $x\in [a,b]$ (cf. Figure 5). Normalize such that $b=0<c\leq y\leq d<a<x=\infty$ and $c<d$. Let $T(z)=e^{-m}(z-y)+y$. By definition of earthquake $E$, we have
$$
f(x,y)=L([a,T(b)]\times [T(c), d])
$$
which gives
$$
f(x,y)=\log\frac{[a-e^{-m}c-(1-e^{-m})y][d-(1-e^{-m})y]}{(a-d)(e^{-m}c)}.
$$
We have
\begin{equation*}
\begin{split}
\frac{\partial}{\partial y}f(x,y)=\frac{-(1-e^{-m})}{a-e^{-m}c-(1-e^{-m})y} +\frac{-(1-e^{-m})}{d-(1-e^{-m})y}<0.
\end{split}
\end{equation*}

\end{proof}

We prove a proposition extending the above lemma to earthquakes with arbitrary support.

\begin{proposition}
\label{prop:min-max-earthquake}
Let $[a_1,b_1]\subseteq [a,b]$ and $[c_1,d_1]\subseteq [c,d]$ be two nested intervals on $S^1$ with $(a,b)\cap (c,d)=\emptyset$. 

Let $E^{\beta}$ be an earthquake with earthquake measure $\beta$ supported on $[a_1,b_1]\times [c_1,d_1]$. Then
$$
L([a,T_{2}(b)]\times [T_{2}(c),d])\leq L(E^{\beta}([a,b]\times [c,d])),
$$
where $T_2$ is a hyperbolic translation with repelling fixed point $b_1$ and attracting fixed point $d_1$ and translation length $m=\beta ([a_1,b_1]\times [c_1,d_1])$.

Let $E^{\gamma}$ be an earthquake with earthquake measure $\gamma$ supported on $[a,b]\times [c,d]$. Then
$$
L(E^{\gamma}([a_1,b_1]\times [c_1,d_1]))\leq L([a_1, T_1(b_1)]\times [c_1,d_1]),
$$
where $T_1$ is a hyperbolic translation with repelling fixed point $a_1$ and attracting fixed point $c_1$ and translation length $m=\beta ([a,b]\times [c,d])$.
\end{proposition}

\begin{proof}
An earthquake $E^{\beta}$ can be approximated by a finite earthquake $E^{\beta_n}$ with support geodesics 
$\{ l_1,l_2,\ldots ,l_{k_n}\}$ in $[a_1,b_1]\times [c_1,d_1]$ and the weights $m_i=\beta_n(l_i)$ for $i=1,2,\ldots ,k_n$ that satisfies (cf. Thurston \cite{Th}, and Gardiner, Hu and Lakic \cite{GHL})
$$
\Big{|} \beta ([a_1,b_1]\times [c_1,d_1])-\sum_{i=1}^{k_n} m_i\Big{|} <\frac{1}{n}
$$
and
$$
\Big{|}E^{\beta}(z)-E^{\beta_n}(z)\Big{|}<\frac{1}{n}
$$
for all $z\in S^1$. 

The above inequality implies that 
$$
L(E^{\beta_n}([a,b]\times [c,d]))\to L(E^{\beta}([a,b]\times [c,d]))
$$
as $n\to\infty$. On the other hand, by applying Lemma \ref{lem:simple_earthq_estimate} to the support of the finite earthquake $E^{\beta_n}$ we get
$$
L([a,T_{2}^n(b)]\times [T_{2}^n(c),d])\leq L(E^{\beta_n}([a,b]\times [c,d])),
$$
where $T_2^n$ is a hyperbolic translation with repelling fixed point $b_1$,  attracting fixed point $d_1$ and translation length $m_1+m_2+\cdots +m_{k_n}$. The first inequality is established by taking $n\to\infty$.

The proof of the second inequality is done in a similar fashion to the above. We leave it to the interested reader.
\end{proof}

In the following lemma we establish the estimate for Liouville measure of a box of geodesics $Q=[a,b]\times [c,d]$ under simple earthquakes whose support geodesic has endpoints $a$ and $c$. This is the case of the largest increase in Liouville measure as established in the previous lemma.

\begin{lemma}
\label{lem:estimate_L_box_Epath}
Let $Q=[a,b]\times [c,d]$ be a box of geodesics and let $D=dist(l(a,d),l(b,c))$ be the distance between the geodesic $l(a,d)$ with endpoints $a,d$ and the geodesic $l(b,c)$ with endpoints $b,c$. Let $E$ be a simple earthquake with the support $g=l(a,c)$ and measure $m> 0$. Then
$$
m+\log\frac{D^2}{4}\leq L(E([a,b]\times [c,d]))\leq m+L([a,b]\times [c,d]).
$$
\end{lemma}

\begin{proof}
Normalize $E$ to be the identity on the half-plane complement of $g$ which contains $d$. We use the upper half-plane model $\mathbb{H}$ and assume that $a=0$, $b>0$, $c=\infty$ and $d=-1$. A direct computation yields
$$
L(E([a,b]\times [c,d]))=\log (e^mb+1)
$$
which easily give estimate in the statement of the lemma.
\end{proof}

\end{document}